\documentclass[11pt]{article}
\newcommand{\Adresses}{{
\bigskip
\footnotesize

Octave Curmi, AIX MARSEILLE UNIVERSIT\'E, CNRS, CENTRALE MARSEILLE, I2M, FRANCE.

\emph{E-mail address: }octave.curmi@univ-amu.fr
}}
\usepackage{amsmath}
\usepackage{amsthm}
\usepackage[pagebackref]{hyperref}
  \hypersetup{
     colorlinks=true,
     linkcolor=blue,
    citecolor=blue
   }
\usepackage[normalem]{ulem}
\usepackage[refpage]{nomencl}
\makenomenclature
\usepackage[toc,page]{appendix}
\usepackage{tikz}
\usetikzlibrary{matrix}
\usepackage[scale=0.7]{geometry}
\usepackage{amssymb}
\usepackage{accents}
\usepackage[utf8]{inputenc}
\usepackage{float}
\usepackage{etoolbox}
\patchcmd{\thenomenclature}{\chapter*}{\chapter*}{}{}
\usepackage[english]{babel}
\usepackage{blindtext}
\usepackage[T1]{fontenc}
\usepackage{mathtools}
\usepackage[nottoc]{tocbibind}
\usepackage{enumitem}
\usepackage{faktor}
\usepackage{xfrac}
\usepackage{calrsfs}
\usepackage{lineno}  % pour numéroter les lignes
\usepackage{makeidx}
\usepackage{xcolor}
\DeclareMathAlphabet{\pazocal}{OMS}{zplm}{m}{n}

\let\mathcal\pazocal
\let\pazocal\temp
\author{Octave Curmi}
\usepackage{microtype}
\newtheorem*{theo}{Theorem}
\newtheorem*{propo}{Proposition}

\newtheorem{theorem}{Theorem}[section]
\newtheorem{lemma}[theorem]{Lemma}

\newtheorem{prop}[theorem]{Proposition}

\newtheorem{defn}[theorem]{Definition}
\newtheorem{rk}[theorem]{Remark}
\newtheorem{ex}[theorem]{Example}

\newcommand\Tau{\mathcal{T}} 
\newcommand\C{\mathbb{C}}
\newcommand\Z{\mathbb{Z}}
\newcommand\R{\mathbb{R}}
\newcommand\N{\mathbb{N}}

\newcommand\Fan{\mathcal{F}}
\newcommand\Sp{\mathbb{S}}

\newcommand*{\supstar}[1][E]{{\overset{\star}{{#1}}}}

\newcommand\X{X}
\newcommand\Xt{\tilde{X}}
\newcommand\Li{S} %bord de X (link)
\newcommand\Xb{\supstar[X]}

\newcommand\Fh{{\widehat{\Fan}}}
\newcommand\Fant{{\widetilde{\Fan}}}
\newcommand\Fb{{\supstar[\Fan]}}

\newcommand\Sk{\pazocal{S}_k}
\newcommand\Sb{{\widetilde{\pazocal{S}}}}
\newcommand\Skt{{\widetilde{\pazocal{S}_k}}}

\newcommand\Sn{{\Skt^{N}}}

\newcommand\Vf{V(f)}
\newcommand\Vft{\widetilde{V(f)}}

\newcommand\Vg{V(g)}
\newcommand\Vgt{\widetilde{V(g)}}

\newcommand\V{V_{f\cdot g}} %% l'union des surfaces Vf et Vg

\newcommand\Supp{\operatorname*{Supp}}
\newcommand\Sing{\operatorname*{Sing}}

\newcommand\rx{r_X}
\newcommand\rt{{\widetilde{r}}}
\newcommand\rb{\supstar[r]}

\newcommand\tr{r_{_{\pazocal{S}}}}
\newcommand\Dt{\mathbb{D}}
\newcommand\Df{\mathbb{D}_{f}}

\newcommand\Dfex{{\mathbb{D}_{f,ex}}}

\newcommand\Do{\mathbb{D}_{0}}%diviseur de la modif de X

\newcommand\Courbes{\pazocal{C}} %Courbes compactes
\newcommand\Courbest{\pazocal{C}_{tot}} %%"toutes les courbes" dans la normalisée
\newcommand*{\Gra}[1][E]{\Gamma(#1)} %Graphes duaux des configurations de surfaces.

\newcommand\Gct{\Gra[\Courbest]}

\newcommand\Gplomb{\Gamma(\partial F)}

\newcommand*{\Grad}[1][E]{\overset{\star}{\Gamma}\left(#1\right)} %Les mêmes, décorés.

\newcommand\Gcdt{\Grad[\Courbest]}

\newcommand\Gmult{\Gamma^\mu (E_{tot})}
\newcommand*{\Gdp}[2]{\Gamma_{#2}({#1})}%Graphe de plombage d'une configuration de surfaces dans une variété lisse.

\newcommand\F{F}  %%fibre de milnor

\newcommand{\ba}[1][E]{\overline{#1}}

\newcommand\Ct{\ba[C]}
\newcommand\Cpt{\ba[C']}

\newcommand\plus{\oplus}
\newcommand\minus{\ominus} %% types d'arêtes / Points triples

\newcommand*{\dual}[1][E]{#1^{\scalebox{.7}{\checkmark}}} %% dual d'un cône
\newcommand\NP{\operatorname*{LNP}}
\newcommand{\LNP}[1][E]{\NP(#1)}

\newcommand\Conv{\operatorname*{Conv}}

\newcommand\Card{\operatorname*{Card}}

\newcommand{\orb}[1][E]{O_{#1}}

\newcommand\sigsing{\sigma_{sing}}

\newcommand*{\vect}[3]{\left(\begin{array}{c}
{#1}\\
{#2}\\
{#3}
\end{array}\right)}

\author{Octave Curmi}
\title{Boundary of the Milnor fiber of a Newton non degenerate surface singularity}
\date{}

\begin{document}

\maketitle

\begin{abstract}
We give an algorithm for the description of the boundary of the Milnor fiber of a Newton non degenerate surface singularity as a graph manifold. This is based on a previous work by the author describing a general method for the computation of the boundary of the Milnor fiber of any reduced non isolated singularity of complex surface.
\end{abstract}

\section{Introduction}

The study of Milnor fibers of complex-analytic functions, which began in the second half of the 20th century with the work \cite{Mil68}, gave rise to a rich interaction between algebra and topology. One of its interesting aspects is that it can be used to provide equations for sophisticated topological objects, such as exotic spheres, see \cite[p 46-48]{Bri00} and \cite{Hirz95}. See also \cite[Section 1]{Sea19}.

In his 1968 book, Milnor explored the topology of isolated singularities of hypersurfaces of $\C^n$, $V(f)=\{f=0\}$, introducing two equivalent fibrations, given respectively by the levels of $f/|f|$ on spheres centered at the critical points of $f$ and by the levels of $f$ in $\C^n$. 

The second of these two fibrations has been extended to more general contexts by Lê, see \cite{Le77}, and is known as the \textbf{Milnor-Lê fibration}. However, the singularities of the ambient space may lead to singular generic fibers. Hamm, in \cite{Ham71}, provided a setting in which the Milnor-Lê fibration is actually a \textbf{smoothing} of $V(f)$, that is, a way to put $V(f)$ in a \textbf{flat} family of analytic spaces with smooth generic fiber. Namely, if $(X,0)$ is a germ of equidimensional complex analytic space, and $f$ is any holomorphic function on $(X,0)$ such that $V(f)$ contains the singular locus $\Sing(X)$ of $X$, then the function $f$ provides a smoothing of the singularity $(V(f),0)$.

It is very hard to describe the whole Milnor fiber up to homeomorphisms, even for isolated singularities, and has been realized for only a few types of singularities. It is the case for the Kleinean singularities A, D, E, where the Milnor fiber is unique and diffeomorphic to a tubular neighbourhood of the exceptional divisor of the minimal resolution (see Brieskorn, \cite{Bri66.2}), as well as for the singularities of normal toric surfaces, with a description by surgery (see Lisca, \cite{Lis08} and Némethi \& Popescu-Pampu, \cite{NemPop10}), and for sandwich singularities (De Jong \& Van Straten, \cite{DejVan98}). As for nonisolated singularities, the only known case is that of hypersurface singularities of the form $\{f(x,y)+zg(x,y)=0\}$, see Sigur\dh sson, \cite{Sig16}. We refer also to the survey \cite{Gre19} for more about the topology of smoothings and deformations of singularities.

On the other hand, the study of the \textbf{boundary} $\partial F$ of the Milnor fiber of a smoothing has been a very active area of research in the last decades. For an isolated singularity, $\partial F$ is unique, and diffeomorphic to the link of the singularity. In today's words, Mumford proved in \cite{Mum61} that the link of any isolated singularity of complex surface is a \textbf{graph manifold} i.e. a manifold describable using a decorated graph whose vertices represent fibrations in $\Sp^1$ over compact surfaces. Waldhausen, in \cite{Wal67}, later introduced this vocabulary and began studying this class of manifolds in itself. Furthermore, thanks to the work of Grauert (\cite{Gra62}), one knows exactly which graph manifolds appear as links of isolated singularities of complex surfaces. However, this strong result is tempered by the fact that one still does not know, for example, which of these manifolds appear as links of singularities of \textbf{hypersurfaces} of $\C^3$.

Still, one would like to get an analogous result for boundaries of Milnor fibers associated to non isolated singularities. The first steps towards the comprehension of the topology of these manifolds were made by Randell \cite{Ran77}, then Siersma in \cite{Sie91}, \cite{Sie01}, who computed the homology of the boundary $\partial F$ of the Milnor fiber in certain cases, and characterized the cases in which $\partial F$ is a rational homology sphere. 

Concerning the general topology of this manifold, a series of results were aimed at proving that the boundary of the Milnor fibers associated to a non isolated singularity is, again, a graph manifold, see the works \cite{MicPic03}, \cite{MicPic04} and \cite{MicPic16} of Michel \& Pichon as well as their work \cite{MicPicWeb07} with Weber for the special case of Hirzebruch surface singularities, and \cite{MicPicWeb09} for the so-called suspensions ($f=g(x,y)+z^n$). In \cite{FerMen}, Fern\'andez de Bobadilla \& Menegon Neto prove the same result for a larger context, and in \cite{NemSzi12}, Némethi \& Szil\'ard give a constructive proof for the case of reduced holomorphic functions $f \colon (\mathbb{C}^3, 0) \to (\mathbb{C}, 0)$. It is this last proof that we extend here and in \cite{cur19.2} to study the boundary of the Milnor fiber of a Newton non degenerate function $f\colon (X_\sigma,0)\to (\C,0)$, where $X_\sigma$ is the $3$-dimensional normal toric variety associated to a cone $\sigma$, and $\Vf\supset \Sing(X)$. In this setting, we adapt the general procedure of \cite{cur19.2} to produce an algorithm for the description of the boundary of the Milnor fiber of $f$, extending the work of Oka in \cite{Oka87} for non degenerate isolated singularities in $\C^3$.

We prove the following theorem, which appeared for the first time in our thesis \cite{cur19}:
\begin{theo}\label{th torique intro} Let $(X,0)$ be the germ at its vertex of a $3$-dimensional toric variety defined by a $3$-dimensional cone in a weight lattice. Let $f \colon (X,0) \to (\mathbb{C}, 0)$ be a Newton-nondegenerate function whose zero locus contains $\Sing(X)$. Then the boundary $\partial F$ of its Milnor fiber is a graph manifold determined by the local Newton polyhedron of $f$. 
\end{theo}

The purely combinatorial nature of the description we make of the manifold $\partial F$ opens the way for the computation of a great number of examples through a future implementation in a computer program. But it also calls for more theoretic work, such as for example extending  what is done in \cite{BraNem}, where Braun \& Némethi do the opposite work, retrieving a possible Newton polyhedron of a function $f : \mathbb{C}^3 \to \mathbb{C}$ having a given graph manifold as boundary of its Milnor fiber, under the hypothesis that this manifold is a rational homology sphere. Another use of this method would be to answer the widely open question of which manifolds can appear as boundaries of Milnor fibers of non degenerate surface singularities. In this direction, our algorithm already provides the following obstruction (see Proposition \ref{prop planar}):
\begin{propo}
The normal form of the plumbing graph of the Milnor fiber of a Newton non degenerate singularity of complex surface is planar.
\end{propo}

The article is organized in the following way:

\begin{itemize}
\item In Section \ref{section reminder}, we recall the main result of \cite{cur19.2}, on which we base the algorithm described in this work,
\item in Section \ref{section toric}, we give the definitions and results of toric geometry which are involved in our construction,
\item Section \ref{section construction graphe} is dedicated to the construction of the intermediate graph $\Gcdt$, which is one of the main ingredients for the description of the manifold $\partial F$ in the general case,
\item in Section \ref{section sufficiency}, we show why the data of $\Gcdt$ is sufficient to proceed to the rest of the computation.
\end{itemize}

I would like to thank Patrick-Popescu Pampu for his support and for sharing his vision of toric geometry, as well as Anne Pichon for her careful reading and remarks.

I thank also the Institut de Mathématiques de Marseille and the project Lipschitz geometry of singularities (LISA) of the Agence Nationale de la Recherche (project ANR-17-CE40-0023) for their financial support which allowed me to attend many enriching events.

\section{Milnor fiber boundary of a non-isolated singularity}\label{section reminder}
We recall here some of the main material of \cite{cur19.2}, concerning the structure of the boundary of the Milnor fiber of the smoothing of an \emph{arbitrary} reduced complex surface singularity. One can consult this article and the references therein for more details.

\subsection{Milnor fibration on a singular variety}
In the sequel, a {\em variety} will mean a reduced and equidimensional complex analytic space. 

Let $(X,0)$ be a germ of complex variety of dimension $3$, and let $f \colon (X,0)\rightarrow (\mathbb{C},0)$ be a germ of holomorphic function on $(\X,0)$ assumed not to divide zero, that is, not to vanish on any irreducible component of $(X,0)$. The function $f$ defines a germ $(\Vf,0)$ of hypersurface on $(\X,0)$, where $\Vf:=\{x\in \X, f(x)=0\}$.
% Denote by $\Sigma \Vf$ the singular locus of $\Vf$, and by $\Sigma \X$ the one of $X$. 

In the sequel, if $(X,0)$ is a germ, $X$ will denote a representative of this germ.

\begin{theorem}
\emph{(H. Hamm, \cite[Satz 1.6]{Ham71}, Lê \cite[Theorem 1.1]{Le77})}\label{thm fibration}

\noindent Given a real analytic function $\rho \colon (X,0)\to (\R_+,0)$ such that $0$ is isolated in $\rho^{-1}(0)$, and $\varepsilon>0$, denote $\X_{\varepsilon}:=\X\cap \rho^{-1}([0,\varepsilon])$, and $\Li_\varepsilon:=\X\cap \{\rho=\varepsilon\}$. Let $f\colon (X,0)\rightarrow (\C,0)$ be a germ of holomorphic function, such that $\X\setminus \Vf$ is smooth. Then there exists $\varepsilon_0>0$, such that $\forall ~ \, 0<\varepsilon \leqslant \varepsilon_0, \exists ~ \delta_\varepsilon > 0$ such that $\forall ~  \,0<\delta\leqslant \delta_\varepsilon$, the map $$f\colon \left\{|f|=\delta\right\} \cap \X_\varepsilon\rightarrow D_\delta$$ is a smooth fibration, where $D_\delta$ denotes the closed disc of radius $\delta$ around $0$ in $\C$.
\end{theorem}

\begin{defn}
The map defined in the previous Theorem is called the \textbf{Milnor fibration} associated to $f$, and its fiber is called the \textbf{Milnor fiber} of the germ of function $f$.
\end{defn}

\begin{rk}
The Milnor-Lê fibration is also sometimes referred to as the Milnor fibration.

Using transversality arguments, one may show that the diffeomorphism type of the Milnor fiber does not depend on the chosen representative, so we speak about \textbf{the Milnor fiber of the germ of function $f\in \mathcal{O}_{\X,0}$}
\end{rk}

\subsection{Graph manifolds}
We explain here the definition of graph manifolds (also called plumbed manifolds) that we refer to. For details, one can consult the foundational articles \cite{Mum61} and \cite{Wal67}, as well as the article \cite{Neu81} for a description of the so-called plumbing calculus and other topological considerations on graph manifolds.

Recall that the Euler number of a fibration in $\Sp^1$ over a surface is an integer $e$ that characterizes the total space of the fibration up to homeomorphism. See \cite[Example 4.6.5]{GomSti} for more details.

\begin{defn}
An \textbf{orientable plumbing graph} is a graph $\Gamma$ with decorations of the following type:
\begin{itemize}
\item Each edge is decorated by a $\plus$ or $\minus$ symbol.
\item Each vertex $v$ is decorated by an Euler number $e_v\in \Z$ and a genus $[g_v]$, $g_v\in \N$. 
\end{itemize}
\end{defn}

\begin{rk}
When representing plumbing graphs, we may omit the decorations $\plus$ and $[0]$.
\end{rk}

\begin{defn}{\bf (Graph manifold)}

If $\Gamma$ is an orientable plumbing graph, define the oriented \textbf{graph manifold} $M_\Gamma$ associated to $\Gamma$ in the following way: for each vertex $v$ of $\Gamma$ decorated by $([g_v],e_v)$, let $M_v$ be the $\Sp^1$-fibration with Euler number $e_v$ over the closed smooth surface $B_v$ with genus $g_v$. Choose an orientation of the base and of the fibers so that, taken together, they give the orientation of $M_v$.

Now, let $\lambda_v$ be the number of times the vertex $v$ appears as endpoint of an edge. Remove from $B_v$ disjoint open disks $(D_i)_{1\leqslant i\leqslant\lambda_v}$, consequently removing as many open solid tori from $M_v$. Each $\partial D_i$ is oriented as boundary component of $B_v \setminus \bigsqcup D_i$. Denote by $M_v^b$ the resulting circle bundle with boundary. The boundary $\partial M_v^b$ is a disjoint union of tori, denote $\partial M_v^b=:\bigsqcup T_i$.

For every edge between the vertices $v$ and $v'$, glue the manifolds $M_v^b$ and $M_{v'}^b$ in the following way: pick boundary components $T=\partial D\times \Sp^1$ in $M_v^b$, $T'=\partial D' \times \Sp^1$ in $M_{v'}^b$, and glue $T$ and $T'$ according to the matrix $\epsilon\begin{bmatrix}
0 & 1\\
1 & 0 
\end{bmatrix}$, $\epsilon$ being the sign on the edge.

Finally, we use the convention $$M_{\Gamma_1 \bigsqcup \Gamma_2}=M_{\Gamma_1} \# M_{\Gamma_1}.$$
\end{defn}

In \cite{Neu81}, the author describes the so-called \textbf{plumbing calculus}, which consists in a list of transformations one can apply on a given plumbing graph without changing the associated graph manifold. Although there are infinitely many different graphs encoding the same graph manifold, Neumann provides the definition of the \textbf{normal form} of a given plumbing graph, as well as an algorithm to obtain it

\subsection{Main aspects of the general result}
We recall here the main steps of the general proof, and show how they relate to the different steps of the computation shown in Sections \ref{section construction graphe} and \ref{section sufficiency} in the setting of toric geometry. We start by introducing a second function $g\colon (X,0)\to (\C,0)$ generic relatively to $f$ in the following sense:

\begin{defn} \label{companion} 
We say that the holomorphic function $g\colon (X,0)\rightarrow (\C,0)$ is a \textbf{companion of $f$} if it satisfies the following conditions:
\begin{enumerate}
\item The surface $\Vg$ is smooth outside the origin. \label{g smooth}
\item $\Vg\cap \Sigma \Vf = \{0\}$. \label{intersection non sauvage}
\item The surfaces $\Vg$ and $\Vf$ intersect transversally outside the origin.\label{intersection transverse}
\item \label{discr} The discriminant locus $\Delta_{(f,g)}$ is an analytic curve.
\end{enumerate}
\end{defn}

In \cite[Lemma 4.4]{cur19.2}, we prove that the restriction to $X$ of a generic linear form on $(\C^N,0)\supset (X,0)$ is a companion of $f$. Using this companion, one defines, for $k\in 2\N$, the germ $$(\Sk,0):=(\{f=|g|^k\},0).$$

This is a germ of $4$-dimensional real analytic variety, and 

\begin{prop}(\cite[Proposition 11.3.3 for $X$ smooth]{NemSzi12}, \cite[Proposition 4.13]{cur19.2})

If $k$ is large enough, $(\Sk,0)$ is an isolated singularity, and its link $\partial \Sk$ is diffeomorphic to the boundary $\partial F$ of the Milnor fiber of $f$.
\end{prop}

This point on, the goal is therefore to describe the manifold $\partial \Sk$. This is done in analogy with the case of an isolated singularity of complex surface, by exhibiting a resolution $\Pi \colon (\Sb,E)\to (\Sk,0)$ of $\Sk$ whose exceptional locus is a transversal union of smooth orientable closed real analytic surfaces in a $4$-dimensional oriented real analytic manifold $\Sb$, and using Theorem \ref{thm corresp boundary plumbing graph}.

However, the existence of such a resolution is not obvious, because for arbitrary germs of 4-dimensional real analytic varieties, we expect the preimage of the origin to be $3$-dimensional. We build an \emph{ad hoc} resolution in several steps.
\begin{itemize}

\item The first step is a modification $\rx \colon (\Xt,\rx^{-1}(0)\to (X,0)$ of the ambient germ which is adapted to the pair $(f,g)$ in the following sense:

\begin{defn}\label{def adapted modif} A modification $\rx\colon (\Xt,\rx^{-1}(0)) \rightarrow (X,0)$ is said to be \textbf{adapted} to the morphism $(f,g)$, where $g$ is a companion of $f$, if the two following conditions are simultaneously satisfied:
\begin{enumerate}
\item \label{modif locus} The modification $\rx$ is an isomorphism outside of $(\Vf\setminus \Vg) \cup \{0\}$. Or equivalently, $\rx$ may not be an isomorphism only at $\Vf\setminus \Vg$ or $\{0\}$.

\item \label{sncd} Denote $\Dt:=\rx^{-1}(\V)$.
Denote by $\Do$ the union of the irreducible components of $\Dt$ sent on the origin by $\rx$. Denote by $\Dfex$ the union of the components of $\mathbb{D}$ sent on curves in $V(f)$, and by $\Vft$ and $\Vgt$ the strict transforms of $\Vf$ and $\Vg$ by $r_X$ respectively. Finally, denote $\Df:=\Dfex \cup \Vft$.

We define two complex curves in $\Xt$ by
$$\Courbest:=\left(\Df \cap \Do\right) \cup (\Df \cap \Vgt).$$ and $$\Courbes:=\left(\Df \cap \Do\right)\cup \left(\Dfex \cap \Vgt\right),$$ the union of the compact irreducible components of $\Courbest.$ 

With these notations, the second condition imposed on $\rx$ is that the total transform $$\Dt:=\rx^{-1}(\V)$$ of $\V$ shall be a simple normal crossings divisor at every point of $\Courbest$.
\end{enumerate}
\end{defn}

In these conditions, define $\Skt$ to be the strict tranform of $\Sk$ by $\rx$. We obtain a modification $$\tr \colon (\Skt,\Courbes)\to (\Sk,0).$$

In this work, this first modification will be built in Section \ref{section construction graphe}, by refining the cone $\sigma$ corresponding to the ambient variety $X=X_\sigma$.

\item Then one needs to resolve the singularities of $\Skt$, which are located along $\Courbes$. This is done by normalizing $\Skt$, and then comparing locally the normalization $\Sn$ of $\Skt$ with the normalization of a suitable singularity of complex surface, to resolve the singularities that remain after normalization. These two steps correspond to the modification of the graph $\Gcdt$ made in Section \ref{section sufficiency}.

\end{itemize}

The local equations of $\Skt$ along points of $\Courbest$ are determined by multiplicities of the pullbacks of $f$ and $g$ to $\Xt$. If $D_i$ is an irreducible component of $\Dt$, denote respectively $m_i,n_i$ the multiplicities of $f\circ \rx$ and $g\circ \rx$ on $D_i$. The following decorated graph will contain all the necessary data for the description of the exceptional locus $E$ of the resolution of $\Sk$:

\begin{defn}\label{def gcdt}
Denote $\Gct$ the dual graph of the configuration of complex curves $\Courbest$. The decorated graph $\Gcdt$ is obtained from $\Gct$ in the following way:
\hfill
\begin{itemize}
\item If $C$ is an irreducible component of $D_1\cap D_2$, where $D_1\in \Df$ and $D_2\in \Do\cup \Vgt$, decorate $v_C$, the vertex corresponding to $C$, with the triple $(m_1;m_2,n_2)$, and with its genus $[g]$, in square brackets. If $D_1\in \Vft$ and $D_2 \in \Vgt$, then the vertex associated to the non-compact curve $C$ is an arrowhead.
\item Decorate each edge $e_p$ of $\Gct$ corresponding to a double point $p$ of $\Courbest$ with $\plus$ if this point is on exactly one component of $\Df$, and with $\minus$ if it is on exactly two different components of $\Df$.

\end{itemize}
\end{defn}

\section{Tools of toric geometry}\label{section toric}
We give in this Section the necessary material to proceed to Section \ref{section construction graphe} and the construction of $\Gcdt$. What follows is mainly based on the books \cite{Oka97} of Oka and \cite{CoxLitSch} of Cox, Little, Schenk. See also the foundational article \cite{Oka87}, where the author computes the link of a Newton non degenerate isolated singularity of surface.

\subsection{Fans, toric varieties and toric modifications}
We start with the definition of the affine variety associated to a cone and to a fan, their local properties, and conclude this subsection by discussing the modification associated to a refinement of a fan.

\begin{defn}\label{def cones}
In this work, a $3$-dimensional \textbf{lattice} $N$ is a free group isomorphic to $(\Z^3,0)$. Define the \textbf{integral length} of an element $u\in N$ as $$l(u)=max\{n\in \mathbb{N},\exists ~ v\in N\setminus \{0\} \text{ such that } u=n\cdot v\}.$$ An element of $N$ is called \textbf{primitive} if its integral length is equal to $1$.

A \textbf{cone} $\sigma$ in $N$ is a subset of $N_\R:=N\otimes_\Z \R\simeq \R^3$ of the form $$\langle v_1,\dots,v_k \rangle := \{r_1\cdot v_1+\dots +r_k\cdot v_k, r_1,\dots ,r_k \in \R^+\}\subseteq N_\mathbb{R}$$ for some $v_1,\cdots,v_k \in N$.

A cone $\sigma$ is called \textbf{strongly convex} if it contains no linear subspace of $N_\R$.

A strongly convex cone is called \textbf{regular} if it can be generated by a family of elements of $N$ that can be completed in a basis of the $\Z$-module $N$.
\end{defn}

\begin{defn}
Denote by $M=\dual[N]:=Hom(N,\mathbb{Z})\simeq \mathbb{Z}^3$ the \textbf{dual lattice} of $N$. For $\sigma$ a convex cone of dimension $d$ in $N_{\mathbb{R}}$, its \textbf{dual cone} is $$\dual[\sigma]:=\{m\in M_{\mathbb{R}}, \forall ~ n \in \sigma, \langle m,n\rangle  \geqslant 0\}\subset M_{\mathbb{R}}$$ and its \textbf{orthogonal space} is $$\sigma^\perp :=\{m\in M_{\mathbb{R}}, \forall ~ n \in \sigma, \langle m,n\rangle = 0\}\subset \dual[\sigma],$$ where $\langle m,n \rangle$ denotes the pairing of $m\in M$ and $n\in N$.
\end{defn}

Note that if $\sigma$ is of dimension $3$, then $\sigma$ is regular if and only if $\dual[\sigma]$ is regular.

\begin{defn}
If $\sigma$ is a strongly convex cone in $N$, then $$X_\sigma:=Spec\left(\C[\dual[\sigma]\cap M]\right)$$ is the affine toric variety associated to it.
\end{defn}

Note that $X_\sigma$ is an irreducible complex affine variety of dimension $3$. See Example \ref{ex var} for an example of affine toric variety.

\begin{rk}\label{plongement var torique affine}
A family $m_1,\cdots,m_k \in M$ generating the semigroup $\dual[\sigma]\cap M$ provides an embedding $X_\sigma \hookrightarrow Spec\left(\C[\chi^{m_1},\cdots,\chi^{m_k}]\right)={\C^k}$, where $\chi^{m}$ denotes the element of $\C[M]$ associated to $m\in M$.
\end{rk}

%\begin{prop} 
%Closed points of $X_{\sigma}$ correspond to semigroup morphisms $(S_\sigma,+) \rightarrow (\mathbb{C},\cdot)$. Denote $\Tau_N$ the set of those morphisms whose image is contained in $\C^*$. The set $\Tau_N$ is (non canonically) isomorphic to $(\C^*)^3$ and dense in $X_\sigma$. This group is called the (algebraic) \textbf{torus} associated to $N$. The action of the group $\Tau_N$ on itself extends to a continuous action on the whole variety $X_{\sigma}$, making $\Tau_N$ the unique $n$-dimensional orbit.
%\end{prop}
%
%On closed points, the action is defined as on $\Tau_N$, by multiplication: if $\phi\in X_\sigma$ and $\psi\in \Tau_N$, then $\phi\cdot \psi\in X_\sigma$ is defined by $$\forall ~ m\in M,(\phi \cdot \psi) (m)=\phi(m)\cdot \psi(m).$$

\begin{defn}
Let $\sigma$ be a cone. A \textbf{face} of $\sigma$ is any cone of the form 

$$\sigma \cap m^{\perp}=\{n \in \sigma, \langle m,n\rangle =0\}, \text{ for some m}\in \dual[\sigma].$$ The fact that $\tau$ is a face of $\sigma$ is denoted $\tau \preceq \sigma$. 
\end{defn}

\begin{prop}\label{prop modif inclusion}(See \cite[Proposition 1.3.16]{CoxLitSch}.) If $\tau \subset \sigma$, then the inclusion $\C[\dual[\sigma]] \subset \C[\dual[\tau]]$ gives rise to a canonical birational morphism of algebraic varieties $X_\tau \rightarrow X_\sigma$. If $\tau$ is a face of $\sigma$, then this morphism is an inclusion.
\end{prop}

\begin{proof}[Proof of the birationality]
The cones $\sigma, \tau$ are strongly convex, hence the dual cones $\dual[\sigma],\dual[\tau]$ are $3$-dimensional, and the fraction fields $\C(\dual[\sigma]\cap M),\C(\dual[\tau]\cap M)$ are both equal to $\C(M)$.
\end{proof}

The following definition will allow us to define non-affine toric varieties:
\begin{defn}\label{def fan}
A \textbf{fan} $\Fan$ in $N_{\mathbb{R}}$ is a finite set of strongly convex cones such that:
\begin{enumerate}
\item If $\sigma \in \Fan$, any face of $\sigma$ is in $\Fan$.
\item The intersection of two cones of $\Fan$ is a face of each.
\end{enumerate}
\smallskip
The \textbf{support} $|\Fan|$ of the fan $\Fan$ is the union $\bigcup \limits_{\sigma \in \Fan} \sigma$ of the cones composing it.
\end{defn}

Note that a cone, together with the collection of its faces, defines a fan. 

\begin{defn} 
A fan $\Fan$ in $N_{\mathbb{R}}$ defines a \textbf{toric variety} $X_\Fan$ in the following way: take the disjoint union of the $X_\sigma$'s, for all $\sigma$ in $\Fan$, and, if $\sigma$ and $\sigma'$ are cones of $\Fan$, glue $X_\sigma$ and $X_{\sigma'}$ along $X_{\sigma \cap \sigma'}$.

$$X_\Fan := \faktor{\left(\bigsqcup \limits_{\sigma \in \Fan} X_\sigma\right)}{\left(X_\sigma \underset{X_{\sigma\cap \sigma'}}{\sim} X_{\sigma'}\right)}$$
\end{defn}

The following key proposition expresses some of the main relations between the combinatorial properties of $\Fan$ and the geometric properties of $X_\Fan$.

\begin{prop}\label{prop orb comb}(See \cite[Theorems 1.3.12 and 3.2.6]{CoxLitSch}.) If $X_\Fan$ is the toric variety associated to the fan $\Fan$, there is a group $\Tau_N$, isomorphic to the algebraic torus $(\C^*)^3$, acting on $X_\Fan$, such that the action of $\Tau_N$ on $X_\Fan$ has the following properties:
\begin{enumerate}
\item There is a $1-1$ correspondance between cones of $\Fan$ and orbits of $X_\Fan$.
\item The orbit $O_\sigma$ associated to a $d$-dimensional cone $\sigma$ of $\Fan$ is an algebraic variety isomorphic to $(\C^*)^{3-d}$.
\item \label{point orb closure} The closure $\overline{O_\sigma}$ of the orbit $O_\sigma$ of $X_\Fan$ is equal to $$\overline{O_\sigma}=\bigcup\limits_{\sigma \preceq \tau} O_\tau.$$
\item The variety $X_\Fan$ is non-singular at every point of $O_\sigma$ $\Leftrightarrow$ it is non-singular at one of its points $\Leftrightarrow \sigma$ is a regular cone, in the sense of Definition \ref{def cones}.
\end{enumerate}
\end{prop}

\begin{rk}
In particular, $O_{\{0\}}=\Tau_N$ and this orbit is dense in $X_\Fan$. Furthermore, if $\Fan=\sigma$, then $O_\sigma$ is the unique minimal-dimensional orbit in $X_\sigma$. If $dim(\sigma)=3$, it is called the \textbf{origin} of $X_\sigma$, and denoted $0_{X_\sigma}$.

With this notation, if $\sigma$ is of maximal dimension, we get the \textbf{germ of toric variety} $(X_\sigma,0_{X_\sigma})$, or $(X_\sigma,0)$, whose local ring of germs of holomorphic functions is $$\mathcal{O}_{X_\sigma,0}=\C\{\dual[\sigma]\cap M\},$$ where $\mathbb{C}\{\dual[\sigma] \cap M\}$ denotes the $\C$-algebra of convergent series with complex coefficients and exponents in the additive semigroup $\dual[\sigma] \cap M$.
\end{rk}

One can encode certain modifications of a given variety $X_\Fan$ through a refinement of $\Fan$:
\begin{defn}\label{def refinement}
A \textbf{refinement} of a fan $\Fan$ in $N_{\mathbb{R}}$ is another fan $\Fan'$ in $N_{\mathbb{R}}$ such that: $$|\Fan|=|\Fan'| \text{ and }\forall ~ \sigma' \in \Fan', \exists ~ \sigma \in \Fan \text{ such that } \sigma' \subset \sigma.$$
\end{defn}

The following \emph{ad hoc} definition is a matter of vocabulary and is introduced for pure reasons of convenience in the exposition made in Section \ref{section construction graphe}.

\begin{defn}
Let $\Fan'$ be a refinement of a fan $\Fan$. If $\tau$ is a cone of $\Fan'$, we call \textbf{minimal containing cone} of $\tau$ in $\Fan$ the minimal-dimensional cone of $\Fan$ containing $\tau$.

If $\tau,\sigma$ are two cones such that $\tau\subset \sigma$, the \textbf{minimal containing face} of $\tau$ in $\sigma$ is the minimal-dimensional face of $\sigma$ containing $\tau$. A cone $\tau$ is said to be in the interior of $\sigma$ if its minimal containing face in $\sigma$ is $\sigma$ itself.
\end{defn}

\begin{defn}Let $\Fan$ be a fan in $N$, and $\Fan'$ a refinement of $\Fan$. The {\bf toric morphism} $\Pi_{\Fan',\Fan} : X_{\Fan'}\rightarrow X_{\Fan}$ associated to this refinement is obtained by gluing the morphisms given by the inclusions of cones of $\Fan'$ in the cones of $\Fan$, defined in Proposition~\ref{prop modif inclusion}.
\end{defn}

\begin{prop}\label{prop modif refinement}(See \cite[Lemma 3.3.21]{CoxLitSch}.) The morphism $\Pi_{\Fan',\Fan}$ is a modification of  $X_{\Fan}$. 
It has the following combinatorial property: if $\sigma' \in \Fan'$, let $\sigma$ be the minimal containing cone of $\sigma'$ in $\Fan$. Then $\Pi_{\Fan',\Fan}(\orb[\sigma'])\subset \orb[\sigma]$.

%The critical locus $E_{\Fan',\Fan}$ of $\Pi_{\Fan',\Fan}$ is exactly the union $$\bigsqcup\limits_{\tau\in \Fan',\tau\notin \Fan} \orb[\tau] \in X_{\Fan'}$$ of orbits of $X_{\Fan'}$ corresponding to new cones, and the discriminant locus $\Delta(\Pi_{\Fan',\Fan})$ is $$\bigsqcup\limits_{\tau\in \Fan,\tau\notin \Fan'} \orb[\tau] \in X_{\Fan},$$ the union of orbits of $X_\Fan$ corresponding to cones that have been subdivided.
\end{prop}

\subsection{Modification associated to a germ of function}\label{subs modif ass fct}

Let $(X_\sigma,0)$ be the germ of affine normal variety associated to a $3$-dimensional strongly convex cone. 
%We want to study the germ of hypersurface $(V(f),0)\subset (X_\sigma,0)$.
%
%In order to do this, let us introduce the Local Newton Polyhedron of a germ of function on a germ of normal toric variety. 

\begin{defn}
\begin{itemize}
\item Let $f =\sum\limits_{m_i\in \dual[\sigma] \cap M} a_i \chi^{m_i} \in \C \{\dual[\sigma] \cap M\}$. The \textbf{support of $f$} is defined as $\Supp(f):=\bigcup\limits_{a_i \neq 0} \{m_i\} \subset \dual[\sigma] \cap M$.
\item The \textbf{Local Newton polyhedron of $f$ at the origin of $X_\sigma$} is defined as $$\LNP[f]:=\Conv\left(\Supp(f) + \dual[\sigma]\right)$$ where $``+''$ denotes the Minkowski sum in $M$ and $\Conv$ denotes the convex hull of a set in a real vector space.
\end{itemize}
\end{defn}

This definition of Newton polyhedron for general germs may be found in \cite[Definition 5]{Ste14}, or \cite[Definition 8.7]{PopSte13}.

The following lemma is a direct consequence of the definitions, and allows one to easily read some parts of the zero locus of a function defined on a germ of toric variety.

\begin{lemma}\label{lemma zerolocus contains orb}
Let $\tau$ be a face of a cone $\sigma$. The ideal $I(\tau)\subset \C\{\dual[\sigma]\cap M\}$ of functions cancelling on the orbit $O_\tau$ of $X_\sigma$ is made of the functions $f\in \C\{\dual[\sigma]\cap M\}$ such that $\Supp(f)\cap \tau^\perp=\emptyset$.
\end{lemma}

\begin{defn}
Let $v\in \sigma$, and $f\in \C\{\dual[\sigma]\cap M\}$. Define the \textbf{height of $\LNP[f]$ in the direction $v$} to be $$h_v(f):=\min\limits_{m\in \LNP[f]} \langle m,v\rangle\in \R_+.$$
\end{defn}

\begin{defn} \label{def eventail from polyhedron} Let $\sigma$ be a cone in $N$, and $f\in \C\{\dual[\sigma]\cap M\}$. Let $\Delta$ be a face of $\NP(f)$ of dimension $d$. Then the set $$\tau_\Delta:=\{v\in \sigma, \Delta_v=\Delta\}=\left\{v\in \sigma/ \forall ~ m\in \Delta, \langle m,v\rangle = h_v(f)\right\}$$ is a cone of codimension $d$ contained in $\sigma$. The set  $$\Fan_f := \{\tau_\Delta, \Delta \subset \NP(f)\}$$ is the \textbf{fan associated to $f$.}

The morphism $\Pi_{\Fan_f}\colon X_{\Fan_f}\rightarrow X_\sigma$ coming from the refinement of $\sigma$ is called the \textbf{modification of $X_\sigma$ associated to $f$.}
\end{defn}

\begin{rk}
Note that on any cone belonging to the fan $\Fan_f$, the height $h_v(f)$ is linear in the argument $v$. This minimal subdivision of $\sigma$ in domains of linearity of the height function is in fact an alternative definition of $\Fan_f$.
\end{rk}

Let us describe the behaviour of the strict tranform $\Vft$ of $\Vf$ under this modification, and more generally, under any toric modification factorizing through $\Pi_{\Fan_f}$.

Let $\Fan$ be a refinement of $\Fan_f$, and $\tau\subset \sigma$ a cone of $\Fan$. Denote $$\Pi_\Fan\colon \left(X_\Fan, \Pi_\Fan^{-1}(0)\right)\rightarrow (X_\sigma,0)$$ the modification associated to $\Fan$, and $\tilde{f}=f\circ \Pi_\Fan$ the pullback of $f$ by this modification. Denote $$\Delta_\tau:=\left\{m\in \LNP[f],\forall ~ v\in \tau, \langle m,v \rangle =h_v(f)\right\}.$$ 

Note that we may have $\tau\neq \tau'$ and still $\Delta_\tau=\Delta_{\tau'}$. However, the assumption that $\Fan$ is a refinement of $\Fan_f$ ensures that this definition makes sense. In fact, if $\gamma$ is the minimal containing cone of $\tau$ in $\Fan_f$, then $\Delta_\tau=\Delta_\gamma$. 

Note also that $dim(\Delta_\tau)\leqslant codim (\tau)$, and that $\tau \preceq \sigma \Rightarrow \Delta_\sigma \subset \Delta_\tau$.

\begin{lemma}\label{lemma orb int}(See \cite[Assertion 3.6.1]{Oka97}.) If $\Fan$ is a refinement of $\Fan_f$, and $\tau\in \Fan$, then $$\widetilde{V(f)}\cap \orb[\tau]\neq \emptyset \Leftrightarrow \dim(\Delta_\tau) \neq 0.$$
\end{lemma}
Furthermore,
\begin{lemma}\label{lemma mult orbit}(See \cite[Theorem 3.4]{Oka97}.) If $\tau=\langle v \rangle_{\R_+}$ for some primitive vector $v$ in $N$, then the multiplicity of $\tilde{f}$ along $\overline{\orb[\tau]}$ is equal to $h_v(f).$
\end{lemma}
%The following elementary lemma establishes then the final link between the regularity of $\widetilde{V(f)}\cap \orb[\tau]$ and the initial function $f$.
%
%\begin{lemma}
%Let $n\geqslant k$, and $\underline{X}^{\underline{m}}\in \C[X_1^{\pm 1},\cdots,X_n^{\pm 1}]$. Then the polynomial $f\in \C[X_1^{\pm 1},\cdots,X_k^{\pm 1}]$ defines a smooth hypersurface of $(\C^*)^k$ if and only if $\underline{X}^{\underline{m}}\cdot f$ defines a smooth hypersurface of $(\C^*)^n$.
%\end{lemma}

%This motivates the following definition, see also \cite[Definition 5]{Ste2}:
The following definitions describe a condition implying that the strict tranform $\Vft$ of $\Vf$ behaves in a nice way with respect to the modification associated to $\Fan$.

\begin{defn}\label{def suitable}
A function $f\in \C\{\dual[\sigma]\cap M\}$ is called \textbf{suitable} if $V(f)$ does not contain any $2$-dimensional orbit of $X_\sigma$, or equivalently, if $\Supp(f)$ has points in each $2$-dimensional face of $\dual[\sigma]$.
\end{defn}

\begin{defn}
For a face $\Delta$ of $\LNP[f]$, of any dimension, we define the truncation $f_{\Delta}$ of $f$ relatively to the face $\Delta$ as the function obtained by keeping only the terms of $f$ corresponding to points of this face: $$f_\Delta:=\sum\limits_{m_i\in \Delta} a_i\chi^{m_i}.$$
\end{defn}

\begin{defn}\label{def nnd}
\begin{itemize}
\item A suitable germ of function $f\in \C\{\dual[\sigma]\cap M\}$ is said to be \textbf{nondegenerate relatively to a compact face} $\Delta \subset \LNP[f]$ if and only if $V(f_{\Delta})$ is smooth in $(\C^*)^n$.
\item A suitable germ of function is said to be \textbf{Newton-nondegenerate}, or \textbf{NND} if it is nondegenerate \index{Newton-nondegenerate function} relatively to every compact face of its local Newton polyhedron.  
\end{itemize}
\end{defn}

\begin{prop}(See \cite[Theorem 3.4]{Oka97}.) Let $f\in \C\{\dual[\sigma]\cap M\}$ be a NND germ of analytic function on $(X_\sigma,0)$, and $\Fan$ a refinement of $\Fan_f$, such that $X_\Fan$ is smooth along $\orb[\tau]$ and $\Delta_\tau$ is compact. Then the intersection $\widetilde{V(f)}\cap \orb[\tau]$ is smooth, and at any point $p$ of this intersection, $\left(\widetilde{V(f)},p\right)$ is a germ of smooth hypersurface of $X_\Fan$ intersecting $\orb[\tau]$ transversally.
\end{prop}
%
%Denote $E:=\widetilde{V(f)}\cap E_\Fan$. Lemma \ref{preimage of origin} implies that $$E=\bigsqcup\limits_{\Delta_\tau ~ compact} \orb[\tau]\cap \widetilde{V(f)}.$$ Then $$\pi:={\Pi_\Fan}_{|\widetilde{V(f)}} \colon (\widetilde{V(f)},E)\rightarrow (V(f),0)$$ is a modification of the germ $(V(f),0)$, and the previous proposition implies
%
%\begin{cor}
%Let $f$ be a germ of NND function on $(X_\sigma,0)$ and $\Fan$ a refinement of $\Fan_f$ such that any cone $\tau\in \Fan$ such that $\Delta_\tau$ is compact and $dim(\Delta_\tau)\geqslant 1$ is regular. Then the modification $\pi\colon (\widetilde{V(f)},E)\rightarrow (V(f),0)$ is a resolution of the germ $(V(f),0)$.
%\end{cor}

Furthermore, the following will be heavily used in Section \ref{section construction graphe}:
\begin{prop}\label{prop orbits 1 and 2}(See \cite[Assertion 6.2.2 and Corollary 6.4.3]{Oka97}.)

Let $f$ be a Newton-nondegenerate function on $X_\sigma$, and $\Fan$ be a refinement of $\Fan_f$. Let $\tau$ be a cone of $\Fan$, such that $\Delta_\tau$ is compact. Then,
\begin{itemize}
\item If $\dim (\tau)=2$, denote $l(\Delta_\tau)$ the integral length of $\Delta_\tau$ in $M$. Then $$Card \left(\widetilde{V(f)}\cap \orb[\tau] \right)=l(\Delta_\tau).$$ 
\item If $\dim (\tau)=$, denote $i(\Delta_\tau)$ the number of points of $M$ in the \textbf{interior} of $\Delta_\tau$. Then $$C(\tau):=\widetilde{V(f)}\cap \overline{\orb[\tau]}$$ is a smooth curve. Furthermore, if $\dim(\Delta_\tau)=2$, this curve is irreducible and $$g(C(\tau))=i(\Delta_\tau).$$
If $\dim(\Delta_\tau)=1$, $C(\tau)$ is a disjoint union of $l(\Delta_\tau)$ smooth curves of genus $0$. 
\end{itemize}
\end{prop}
\begin{rk}
Recall that, in both cases, if $\dim(\Delta_\tau)=0$, then $\widetilde{V(f)}\cap \overline{\orb[\tau]}=\emptyset$.
\end{rk}
Let us conclude this section with a lemma giving a precise meaning to the genericity of Newton-nondegenerate functions.

\begin{lemma}\label{NND generic}(See \cite[Chapter V, Section 3]{Oka97}.) Let $f$ be a holomorphic suitable function on $X_\sigma$. Then in the space of coefficients of all functions $h$ such that $\NP(h)=\NP(f)$, those which are non-degenerate
are Zariski-dense.
\end{lemma}

Finally, we will need a corollary of the Bernstein-Koushnirenko-Khovanskii Theorem (see \cite[Theorem 7.5.4]{CoxLitOsh}), in dimension $2$. We will need the following notations:

If $f\in \C[M]$ for some lattice $M$ of dimension $2$, denote $P(f):=\Conv(\Supp(f)).$ Denote by $V(P)$ the \textbf{lattice volume} of a polyhedron $P$, for which a simplex is of volume $1$. Finally, if $P_1,P_2$ are two polyhedra in $M$, denote $$V(P_1,P_2)=\dfrac{Vol(P_1+P_2)-Vol(P_1)-Vol(P_2)}{2},$$ called the \textbf{mixed volume} of $P_1$ and $P_2$.
\begin{lemma}\label{BKK dim 2}
Let $M$ be a $2$-dimensional lattice, and let $f\in \C[M]$ be a Newton-nondegenerate function. Denote $P_1:=P(f)$. Then for any convex polyhedron $P_2$ with vertices in $M$, a generic choice of coefficients for the elements of $P_2\cap M$ will provide a function $g$ such that $P(g)=P_2$ and, denoting $P=P_1+P_2$, the compactifications $\widetilde{V(f)}$ and $\widetilde{V(g)}$ of $V(f)$ and $V(g)$ in $X_{\Fan_P}$ intersect transversally only on $\orb[0]$, in $V(P_1,P_2)$ points.
\end{lemma}

\section{Construction of $\Gcdt$}\label{section construction graphe}
Let $(X,0)=(X_\sigma,0_{X_\sigma})$ be the germ of normal $3$-dimensional toric variety associated to a rational strongly convex cone $\sigma\subset N_\R$, where $N$ is a $3$-dimensional lattice. Let $\sigsing$ denote the union of non regular faces of $\sigma$, i.e. the union of faces $\tau$ of $\sigma$ corresponding to orbits $O_\tau$ along which $X$ is singular. 

Let $f \colon (\X,0)\to (\C,0)$ be a holomorphic function such that $\Supp(f)\cap \sigsing=\emptyset$. This guarantees that $\Vf := \{f=0\}\supset \Sing(X)$, and therefore that $f$ is a smoothing of $\Vf$. We also assume that $f$ is a Newton non degenerate function, in particular it is suitable, in the sense of Definitions \ref{def suitable} and \ref{def nnd}.

In this section, we explain how to compute the decorated graph $\Gcdt$ of Definition \ref{def gcdt}	 from the cone $\sigma$ and the Newton polyhedron $LNP(f)$ of $f$. 

We can already make the following remark:

%\begin{rk}
%If $\tau$ is a $2$-dimensional face of $\sigma$ which is not regular, then the condition $\Vf \supset O_\tau$ implies that there is a cone of dimension $1$ in $\Fan_f$ whose minimal containing cone is $\tau$.
%\end{rk}

\begin{rk}\label{rk orb dim 2 f}
Let $\tau$ be a face of $\sigma$ of dimension $2$. Then $\Vf\supset O_\tau$ iff there is a $1$-dimensional cone of $\Fan_f$ whose minimal containg face in $\sigma$ is $\tau$.

In particular, if $\tau$ is not regular, then the condition $\Vf \supset \Sing(X)$ implies that there is a cone of dimension $1$ in $\Fan_f$ whose minimal containing cone is $\tau$.
\end{rk}

\subsection{Universal companion polyhedron}

Although it is impossible to provide a function $g$ that will be a companion of every Newton-nondegenerate function $f$ on $X$, there exists an object which will be compatible with every function $f$.

\begin{defn}
For any sequence $G=(m_1,\cdots,m_k)$ generating the semigroup $\dual[\sigma]\cap M$, denote $$P(G):=Conv\left(\bigcup\limits_{1\leqslant i\leqslant k} m_i + \dual[\sigma]\right).$$
\end{defn}

In the sequel, we fix such a family $G$. 

By Remark \ref{plongement var torique affine}, the sequence $G$ leads to an embedding $(X,0)\hookrightarrow {\C^k}_{x_1,\cdots,x_k}$. From this viewpoint, $P(G)$ is in fact the local Newton polyhedron of the restriction to $X$ of a generic linear form of ${\C^k}_{x_1,\cdots,x_k}$.

Now, the fact that the restriction of a generic linear form is a companion of $f$ implies:

\begin{lemma}
For any germ of function $f$ on $(X,0)$, there is a function $g$ that is a companion of $f$ and such that $$LNP(g)=P(G).$$
\end{lemma}

From now on, fix $g \colon (X,0)\to(\C,0)$ a Newton non degenerate companion of $f$ such that $LNP(g)=P(G)$. 

Remark \ref{rk orb dim 2 f} has the following consequence for $g$:
\begin{rk}\label{rk orb dim 2 g}
There is no $1$-dimensional cone of $\Fan_g$ whose minimal containing face in $\sigma$ is $2$-dimensional.
\end{rk}

In the next subsections, we construct a fan $\Fant$ refining $\sigma$ and such that the associated modification $$\rt:= \Pi_\Fant \colon X_\Fant\rightarrow X$$ is adapted to the pair $(f,g)$ in the sense of Definition \ref{def adapted modif}. We will follow this construction on the following example: 

\begin{ex}\label{ex var}
Denote $M:=\Z^3$, and let $$u_1= \vect 012, u_2= \vect 010, u_3= \vect 11{-1}, u_4= \vect 100$$ be vectors in $\R^3$, and $\sigma=\langle u_1,u_2,u_3,u_4\rangle _{\R_+}$.

Let $X:=X_\sigma$ be the $3$-dimensional toric variety corresponding to $\sigma$. The cone $\sigma$ is not simplicial, hence $X$ is singular at the origin. Furthermore, the face $\tau_{1,2}:=\langle u_1,u_2\rangle_{\R_+}$ of $\sigma$ is singular, hence $X$ is singular along $0_{\tau_{1,2}}$.

The cone $\dual[\sigma]$ is generated by the vectors $u=\vect 101, v=\vect 011, w=\vect 100, x=\vect 02{-1}.$ The face of $\dual[\sigma]$ corresponding to the face $\tau_{1,2}$ of $\sigma$ is the ray generated by $W$. Furthermore, the semigroup $S_\sigma=\dual[\sigma]\cap M$ is generated by $u,v,w,x$ and $y=\vect 010$. 
%By abuse of notation, denote in the same way the corresponding elements of the algebra $\C[\dual[\sigma]\cap M]$. 
The relations between these vectors provide the description $$X_\sigma=Spec\left(\faktor{\C[u,v,w,x,y]}{(y^3-xv,ux-wy^2,vw-uy)}\right)$$

Figure \ref{fig:polgex2} shows the local Newton polyhedron of the restriction $g$ to $X$ of a generic linear form in $\C^5_{u,v,w,x,y}$. Full lines represent the cone $\sigma$, while dashed lines represent the axes of coordinates, left for clarity.

\begin{figure}[h]
\begin{center}

        \includegraphics[totalheight=8cm]{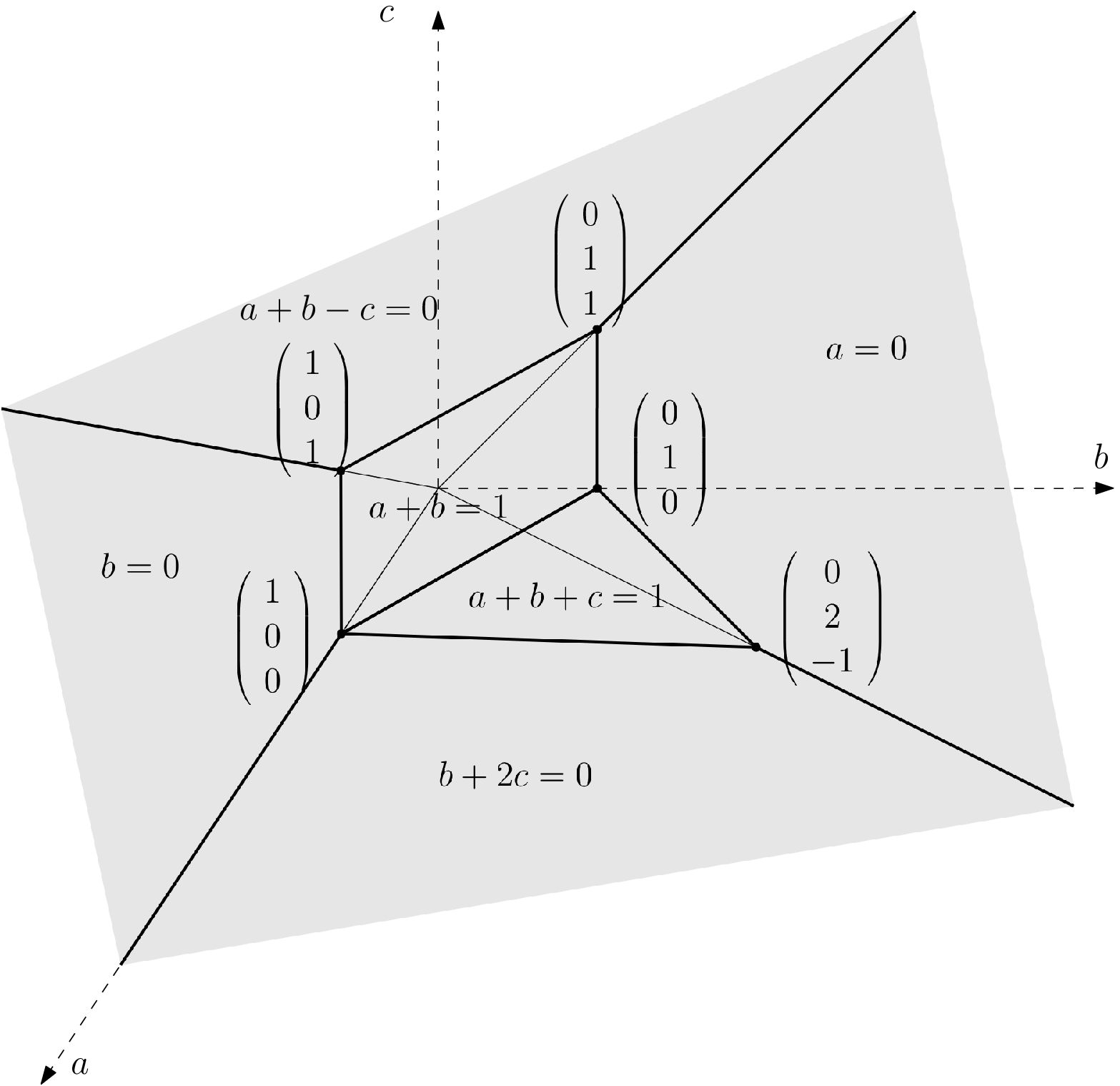}
       \caption{$LNP(g)$.}
    \label{fig:polgex2}
\end{center}
\end{figure}

Consider the function $f\in \C[\dual[\sigma]\cap M]$ given by $$f=\chi^{\left(\begin{smallmatrix}
1\\1\\2
\end{smallmatrix}\right)}+\chi^{\left(\begin{smallmatrix}
2\\0\\1
\end{smallmatrix}\right)}+\chi^{\left(\begin{smallmatrix}
0\\2\\0
\end{smallmatrix}\right)}+\chi^{\left(\begin{smallmatrix}
0\\4\\-2
\end{smallmatrix}\right)}=uv+uw+y^2+x^2.$$

Note that $\Vf\supset \Sing(X)$, since $f$ admits no multiple of $W$ in its support. Figure \ref{fig:polfex2} shows the local Newton polyhedron of $f$, where we kept again the coordinate axis and the cone $\dual[\sigma]$. Black round marks represent the elements of $\Supp(f)$, and white square marks represent the other points of $M$ in the compact faces of $LNP(f)$.

\begin{figure}[h]
\begin{center}

        \includegraphics[totalheight=10cm]{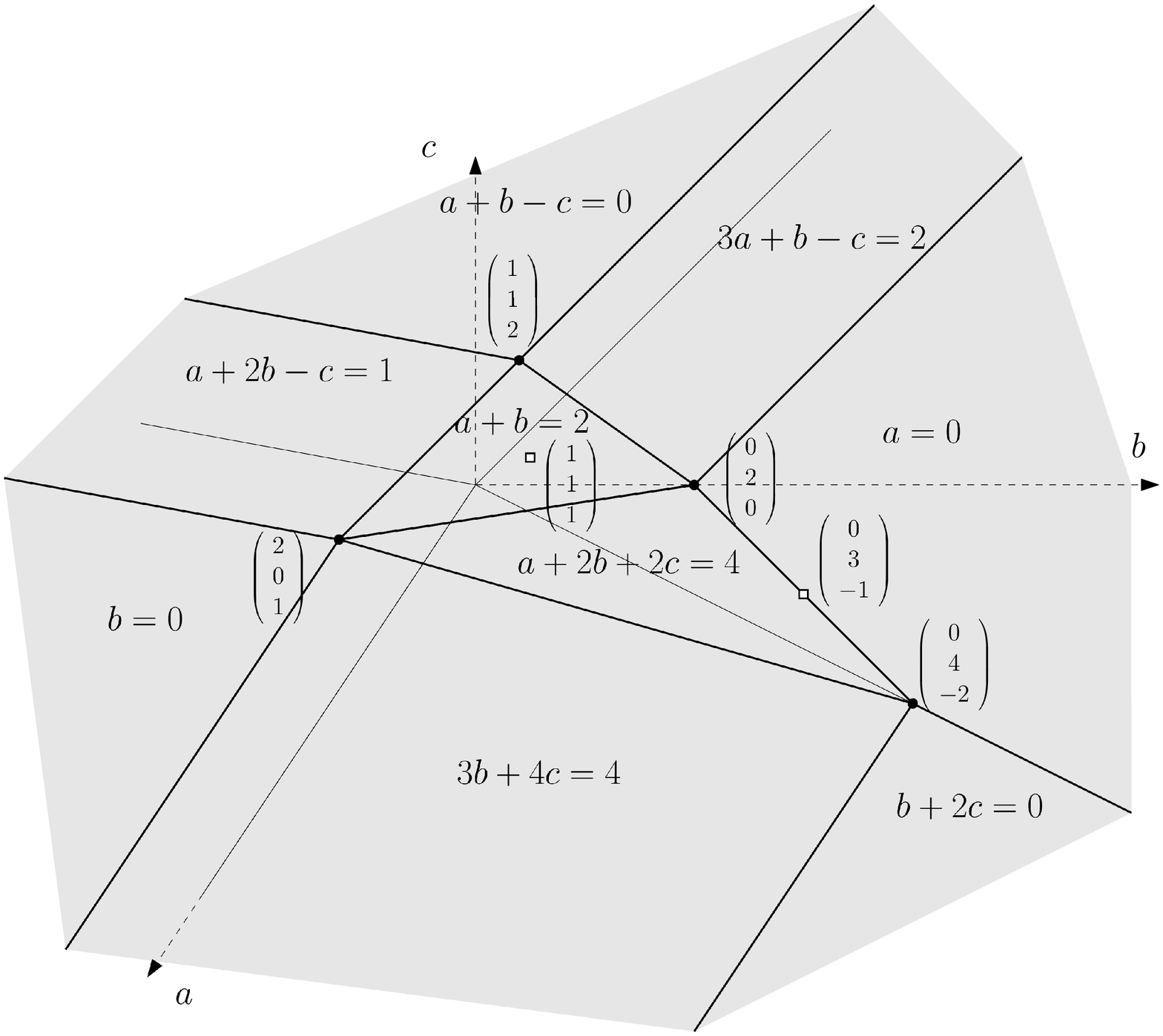}
       \caption{$LNP(f)$.}
    \label{fig:polfex2}
\end{center}
\end{figure}
\end{ex}
\subsection{The modification associated to the polyhedra}

Denote $\Fb:=\Fan_{f\cdot g}$ the fan associated to the germ $f\cdot g$. This fan is the minimal refinement of both $\Fan_f$ and $\Fan_g$. Denote $\Xb=X_{\Fb}$ and $$\rb:=\Pi_{\Fan_{f\cdot g}}\colon \Xb \rightarrow X$$ the modification associated to this refinement. 

Note that we have the following commutative diagram:
\begin{center}

\begin{tikzpicture}
\filldraw
(0,0) circle (0pt) node[] {$X$}
(-2,1) circle (0pt) node[] {$X_{\Fan_f}$}
(2,1) circle (0pt) node[] {$X_{\Fan_g}$}
(0,2) circle (0pt) node[] {$\Xb$}
(-1.4,1.6) circle(0pt) node[] {}
(1.4,0.3) circle(0pt) node[] {};
\path[-stealth]
(-2,0.7) edge node[below]{$\Pi_{\Fan_f}$} (-0.3,0.3)
(2,0.7) edge node[below]{$\Pi_{\Fan_g}$} (0.3,0.3)
(-0.3,1.7) edge node[above]{$\Pi_{\Fb,\Fan_f}$} (-2,1.3)
(0.3,1.7) edge node[above]{$\Pi_{\Fb,\Fan_g}$} (2,1.3)
(0,1.7) edge node[right] {$\rb$} (0,0.3);
\end{tikzpicture}
\end{center}

%Denote $\Dtb:= \rb^{-1}(\V)$, and in the same way, $\Vfb,\Vgb$ the strict transforms of $V(f),V(g)$ by $\rb$. Denote also $\Dfexb$ the union of components of $\Dtb$ whose images by $\rb$ are curves of $\Vf$, and $\Dob:=\rb^{-1}(0)$.
%
%\begin{discussion} {\bf (Reading rules.)}\label{disc reading}
%The irreducible components of $\Dob$ are the closures of orbits $O_\tau$ of $X_\Fb$ corresponding to $1$-dimensional cones $\tau$ of $\Fb$ whose minimal containing face in $\sigma$ is $\sigma$ itself. Such components are intersected by $\Vfb$ (resp. $\Vgb$) if and only if the minimal containing cone of $\tau$ in $\F_f$ (resp. $\F_g$) is of dimension $1$ or $2$.
%
%The components of $\Dfexb$ are the closures of orbits $O_\tau$ of $X_\Fb$ corresponding to $1$-dimensional cones $\tau$ of $\Fb$ whose minimal containing face in $\sigma$ is of dimension $2$. 
%
%\end{discussion}

\begin{ex}
The fact that the fan $\Fan_{f \cdot g}$ is the minimal refinement of both $\Fan_f$ and $\Fan_g$ implies that it can be computed by ``superposing'' the two fans $\Fan_f$ and $\Fan_g$, as in figure \ref{fig:sigmarafpolsnb}. This figure is to be understood as the cone over the plane figure. In this figure, thick lines indicate the $2$-dimensional cones which are contained in $2$-dimensional cones of $\Fan_f$, dashed lines indicate the $2$-dimensional cones which are contained in $2$-dimensional cones of $\Fan_g$, so the thick dashed line indicates a $2$-dimensional cone which is in both. 

By Lemma \ref{lemma orb int}, white marks correspond to $2$-dimensional orbits of $X_\Fb$ whose closures are intersected by the strict transform of $\Vf$ by $\rb$. Proposition \ref{prop modif refinement} implies that square marks correspond to $2$-dimensional orbits of $X_\Fb$ whose the closures are sent by $\rb$ on curves in $\Vf$. Remarks \ref{rk orb dim 2 f} and \ref{rk orb dim 2 g} imply that any $1$-dimensional cone whose minimal containing face in $\sigma$ is $2$-dimensional corresponds to such an orbit.

%Hollow square marks correspond to $2$-dimensional orbits of $X_\Fb$ whose closures are irreducible components of $\Dfexb$. Remarks \ref{rk orb dim 2 f} and \ref{rk orb dim 2 g} imply in fact that any $1$-dimensional cone whose minimal containing face in $\sigma$ is $2$-dimensional corresponds to a component of $\Dfexb$.
\begin{figure}[h]
\begin{center}

        \includegraphics[totalheight=8cm]{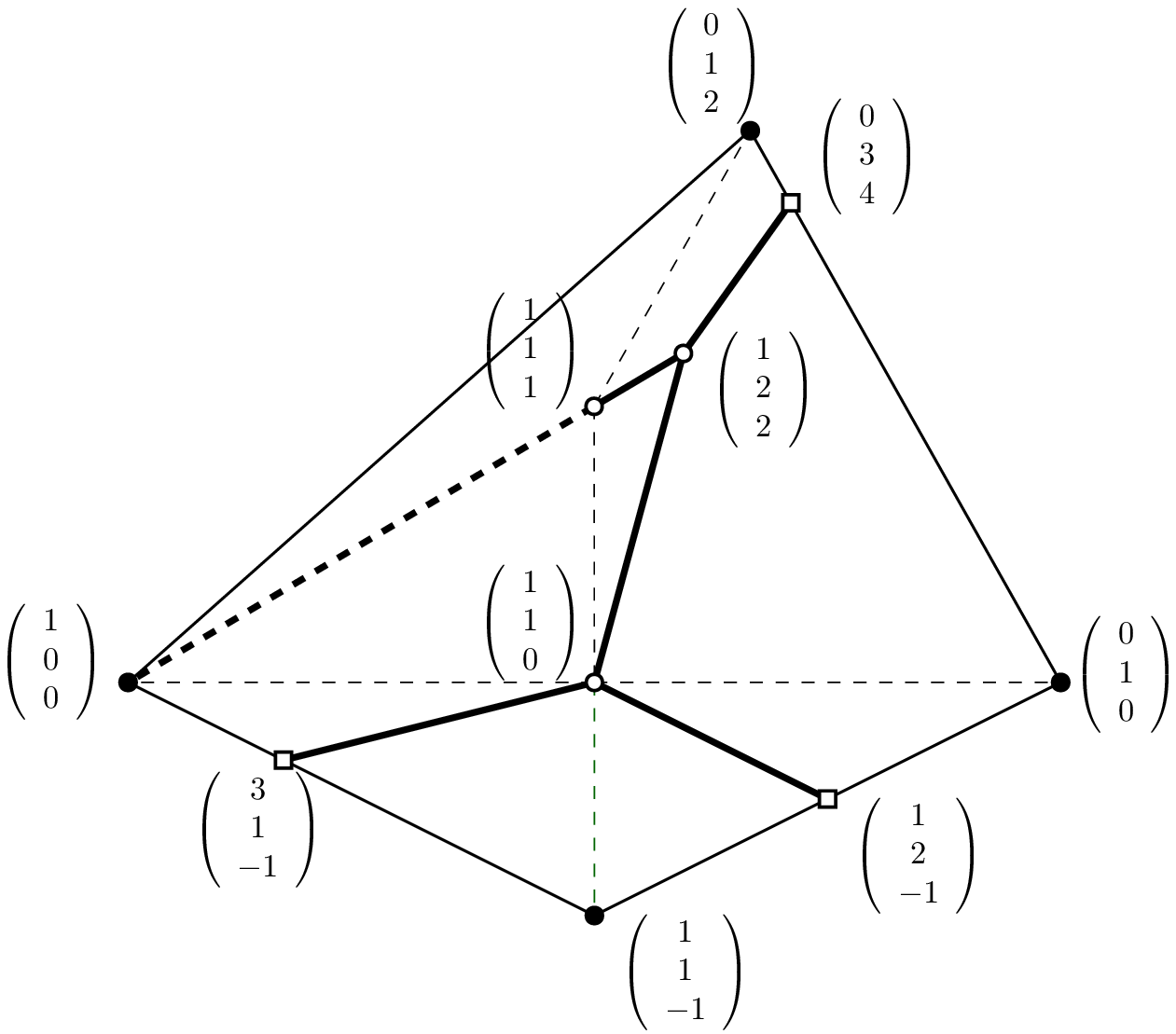}
       \caption{The fan $\Fb$.}
    \label{fig:sigmarafpolsnb}
\end{center}
\end{figure}
\end{ex}

At this step, we do not yet have a modification of $X_\sigma$ respecting all the conditions of Definition \ref{def adapted modif} of an adapted modification. We still need to refine some cones of $\Fb$.

\subsection{First refinement of $\Fb$}
We obtain a new fan $\Fh$ by, first, refining regularly the $2$-dimensional cones of $\Fb$ which are contained in $2$-dimensional cones of $\F_f$. This is done in a canonical way, see for example \cite[Lemma 3.3]{Oka87} or \cite{Pop07}. These cones correspond to $1$-dimensional orbits of $X_\Fb$ which are intersected by the strict tranform of $\Vf$ by $\rb$. Then, one needs to add new $2$-dimensional cones in a non-canonical way in order to get a collection of cones respecting the definition of a fan, see Definition \ref{def fan}. One may have to add new cutting cones at this step.

Denote by $\Fh$ the fan obtained at the end of this process.

\begin{ex}
In our example, these $2$-dimensional cones of $\Fb$ are already regular, so the fan $\Fh$ is the same as $\Fb$.
\end{ex}

%By analogy with the previous subsection, we introduce the notations $\Doh,\Vfh,\Vgh,\Dfexh$, and $\rh \colon X_\Fh \to X_\sigma$.
\subsection{The fan $\Fant$}
%Here again, the components of $\Doh,\Vfh,\Vgh,\Dfexh$ can be read on $\Fh$ as in Discussion \ref{disc reading}, and if $\gamma$ is a cutting cone of $\Fh$, then $\overline{O_\gamma}$ is an irreducible curve of $\Dfexh \cap \Doh$. This is what motivates the final refinement of $\Fh$.

\begin{defn}
We call \textbf{cutting cone} of a fan $\Fan$ with support $\sigma$ any $2$-dimensional cone $\gamma$ of $\Fan$ with $1$-dimensional faces $\tau_1,\tau_2$ so that 
\begin{enumerate}
\item the minimal containing face of $\tau_1$ in $\sigma$ is $\sigma$,

\item and the minimal containing face of $\tau_2$ in $\sigma$ is $2$-dimensional.
\end{enumerate}
\end{defn}

%If $\gamma$ is a cutting cone of $\Fb$, then $\overline{O_\gamma}$ is an irreducible curve of $\Dfexb \cap \Dob$. The denomination \emph{cutting cone} is chosen in reference to the so called ``cutting edges'' introduced in \cite[Definition 7.2.2]{NemSzi12}. 

Denote $\Fant$ a fan obtained by refining the $3$-dimensional cones of $\Fh$ having a cutting cone as a face. 
%These cones correspond to points of $X_\Fh$ which are on curves of $\Dfexh \cap \Doh$, whence the need for $X_\Fant$ to be smooth here. 

There is no canonical way to refine regularly a $3$-dimensional cone, but a first necessary step is the refining of its $2$-dimensional faces. For a description of the general process, see \cite[Paragraph 3]{Oka87}.

Our construction ensures the following

\begin{lemma}
The modification $\rx := \Pi_\Fant \colon \Xt:= X_\Fant \to X$ is adapted to the pair $(f,g)$, in the sense of Definition \ref{def adapted modif}.
\end{lemma}

\begin{ex}
Figure \ref{fig:sigmarafnb} shows a possible $\Fant$ for our example. The new cutting cones are represented by continuous thin lines. We erased one end of the corresponding segments on the drawing in accordance with what will be explained in the next subsection, see Lemma \ref{lemma intersection exc strict}. The other new $2$-dimensional cone is represented with a dotted line.

\begin{figure}[h]
\begin{center}
        \includegraphics[totalheight=8cm]{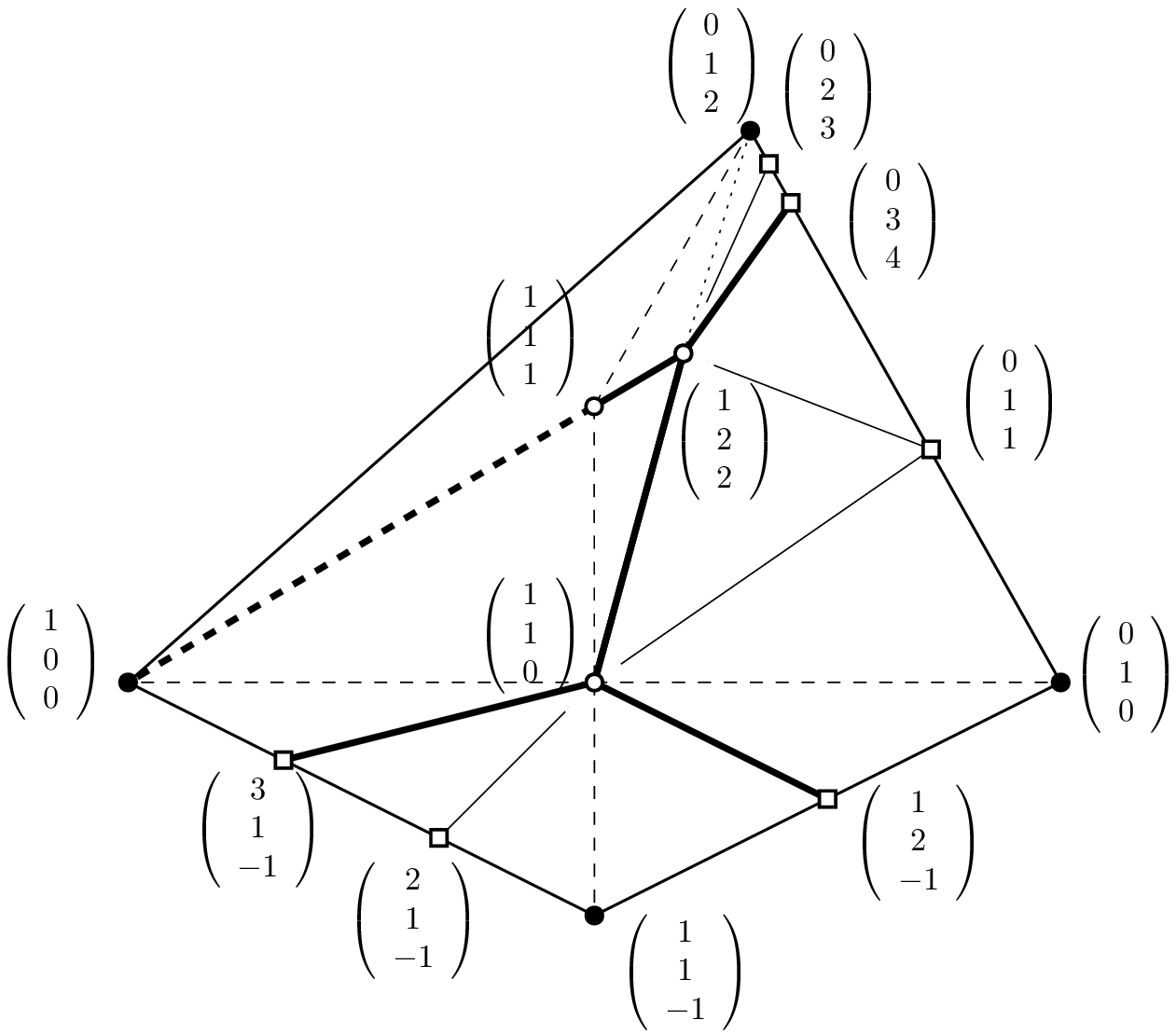}
       \caption{The fan $\Fant$.}
    \label{fig:sigmarafnb}
\end{center}
\end{figure}
\end{ex}

\subsection{Reading $\Gcdt$}

Let us explain now how the decorated configuration $\Courbest$ can be read from the fan $\Fant$, together with $\LNP[f]$ and $\LNP[g]$.

Recall that $\Courbest= \left(\Df\cap \Do\right)\cup \left(\Df\cap \Vgt\right)$ can be decomposed as $$\left(\Dfex\cap \Do\right)\cup \left(\Vft\cap \Do\right) \cup \left(\Vft\cap \Vgt\right) \cup \left(\Dfex\cap \Vgt\right).$$

Since $\Fan_g$ contains no cutting cone, the last intersection is empty. There are therefore three types of curves to be read.

We introduce the following notations:
\begin{defn}For $\tau\in \Fant$, denote repectively $\Delta_\tau(f)$ and $\Delta_\tau(g)$ the corresponding faces of $LNP(f)$ and $LNP(g)$, as in Subsection \ref{subs modif ass fct}.

To make short, if the dimension of those faces is $1$ or less, denote $$l_\tau(f):=l(\Delta_\tau(f)) \text{ and }l_\tau(g):=l(\Delta_\tau(g))$$ with the convention that the length of a $0$-dimensional face is $0$.

If $dim(\Delta_\tau(f))\leqslant 2$, denote $i_\tau$ the number of interior points of $\Delta_\tau(f)$, with the convention that the number of interior points of a face of dimension $0$ or $1$ is $0$.

Finally, a cone $\tau$ in the interior of $\sigma$ is called \textbf{pertinent} if $dim(\Delta_\tau(f))\geqslant 1$, that is, if $\Vft \cap O_\tau \neq \emptyset$. 
\end{defn}

\begin{prop}{\bf (Two types of compact curves.)}
\hfill
\begin{enumerate}
\item \label{point exc curves} \textbf{The first type is made of curves in $\Dfex\cap \Do$}, called \textbf{exceptional curves}. Such curves are of the form $\overline{O_\gamma}$, where $\gamma$ is a cutting cone of $\Fant$.

Two such curves intersect if and only if the corresponding cones are faces of a same $3$-dimensional cone of $\Fant$. The nature of their intersection point can be read on the fan: it is $\plus$ if and only if they are in the same component of $\Dfex$. 

The multiplicity decorations are obtained in the following way: let $\gamma$ be a cutting cone, and $\tau_1, \tau_2$ be its $1$-dimensional faces, with $\tau_2$ in the interior of $\sigma$. Then $\overline{O_\gamma}$ has multiplicity decoration $(h_{\tau_1}(f);h_{\tau_2}(f),h_{\tau_2}(g))$. Each exceptional curve is \textbf{rational}.

\item \label{point str transf curves}\textbf{The second type is made of curves of $\Vft\cap \Do$}, called \textbf{strict tranform curves}. They correspond to $1$-dimensional pertinent cones $\tau$ of $\Fant$.

Furthermore, in these conditions, the intersection $\overline{O_\tau}\cap \Vft$ is 
\begin{itemize}
\item an irreducible curve $C_\tau$ of genus $i_\tau$ if $dim(\Delta_\tau(f))=2$.
\item a disjoint union $C_\tau={C_\tau}^1\bigsqcup \cdots \bigsqcup {C_\tau}^{l_\tau(f)}$ of $l_\tau(f)$ irreducible rational curves if $dim(\Delta_\tau(f))=1$.
\end{itemize}

Each connected component of $C_\tau$ has multiplicity decoration $(1;h_\tau(f),h_\tau(g))$.
%By construction of $\Fant$ and the non-degeneracy of $f$, these curves are smooth, $\Xt$ is smooth along each of them, and they intersect transversally.

Furthermore, let $\tau_1, \tau_2$ be two $1$-dimensional such cones. The possibly disconnected curves $C_{\tau_1}$ and $C_{\tau_2}$ intersect if and only if $\tau_1$ and $\tau_2$ are faces of the same $2$-dimensional pertinent cone $\gamma$. In these conditions, $$Card(C_{\tau_1}\cap C_{\tau_2})=l_\gamma(f)$$ and these intersection points are all of type $\plus$.

In this situation, if $dim(\tau_1)=1$, then $l_{\tau_1}(f)=l_\gamma(f)$, and each connected component of $C_{\tau_1}$ intersects $C_{\tau_2}$ in exactly one point.
\end{enumerate}

\end{prop}

\begin{proof}
\begin{enumerate}
\item If $\gamma$ is a cutting cone, then by Proposition \ref{prop orb comb}, point \ref{point orb closure}, $\overline{O_\gamma}$ is the intersection of the closures of the orbits corresponding to its $1$-dimensional faces $\tau_1,\tau_2$, where $\tau_1$ is in the interior of $\sigma$. Now, Proposition \ref{prop modif refinement} implies that $\overline{O_{\tau_1}}$ is a component of $\Do$, and Remark \ref{rk orb dim 2 f} implies that $\overline{O_{\tau_2}}$ is a component of $\Dfex$.

The multiplicity decorations are obtained thanks to Lemma \ref{lemma mult orbit}.
\item This is a direct application of Proposition \ref{prop orbits 1 and 2}, together with Lemma \ref{lemma orb int} and Proposition \ref{prop orb comb}, point \ref{point orb closure}. Lemma \ref{lemma mult orbit} provides the multiplicities.
\end{enumerate}
\end{proof}

\begin{lemma}\label{lemma intersection exc strict} {\bf (Intersection of curves of the first and of the second type.)}

Let $C_1=\overline{O_\gamma}$ be an exceptional curve, and $C_2=C_\tau$ be a strict transform curve. Then $$C_1\cap C_2 \neq \emptyset \Leftrightarrow \tau \prec \gamma \text{ and } \gamma \text{ is pertinent}.$$

In these conditions, $$Card(C_1\cap C_2)=l_\gamma(f)$$ and each intersection point is of type $\minus$.

Furthermore, if $dim(\Delta_\tau(f))=1$, then $l_\gamma (f)=\cdot l_\tau(f)$, and each connected component of $C_2$ is intersected by $C_1$ in $1$ point.
\end{lemma}

\begin{proof}
If $\tau$ is not a face of $\gamma$, then either $\tau$ and $\gamma$ are not faces of a same cone, in which case $\overline{O_\tau}\cap \overline{O_\gamma}$ is empty, or they are faces of a same $3$-dimensional cone $\delta$, and $\overline{O_\tau}\cap \overline{O_\gamma}= O_\delta$, which is not intersected by $\Vft$, in virtue of Lemma \ref{lemma orb int}.

If $\tau$ is a face of $\gamma$, then $\overline{O_\gamma}\subset \overline{O_\tau}$, and $\Card \left(\overline{O_\gamma} \cap \Vft\right)=l_\gamma (f)$ and by definition of $C_\tau$, each of these points of intersection is a point of $C_\tau$.
\end{proof}

To complete the description of the situation, we need to add the non compact curves of the intersection $\Vft \cap \Vgt$.

Pick generic coefficients of $g$ so that for any $2$-dimensional orbit $O_\tau$ of $\Xt$ intersected by $\Vft$, the truncations $f_{_{\Delta_\tau(f)}}$ and $g_{_{\Delta_\tau(g)}}$ verify the hypothesis of Lemma \ref{BKK dim 2}.

\begin{defn}
Denote $V(\tau)$ the mixed $2$-dimensional volume of $\Delta_\tau(f)$ and $\Delta_\tau(g)$. 
\end{defn}

With this choice of $g$, we can access the rest of the configuration $\Courbest$. Indeed,

\begin{lemma}{\bf (Adding non-compact curves)}

Let $C_\tau$ be a strict transform curve of $\Courbes$, possibly disconnected. Then $$Card\left(C_\tau\cap \Vgt\right)=V(\tau),$$ each of these intersection points being an intersection point of $C_\tau$ with a curve in $\Vft\cap \Vgt$.

Furthermore, if $dim(\Delta_\tau(f))=1$, $\exists ~ k\in \N\text{ such that } V(\tau)=k\cdot l_\tau(f)$, and each connected component of $C_\tau$ is intersected in $k$ points.

Each of these points is of type $\plus$, and, in $\Gcdt$, the new curves are represented by arrowheads decorated with $(1;0,1).$
\end{lemma}

\begin{proof}
This is a consequence of Lemma \ref{BKK dim 2}, and the reduceness of $\Vf$ and $\Vg$ provide the multiplicities $1$ on the components of strict tranforms.
\end{proof}

\begin{ex}
Figure \ref{fig:gcdt} shows the graph $\Gcdt$ of our example. The representation of the graph reflects the disposition of the fan $\Fant$ of figure \ref{fig:sigmarafnb}. Following the classical convention, we do not indicate the decorations $\plus$.

\begin{figure}[h]
\begin{center}
        \includegraphics[totalheight=7cm]{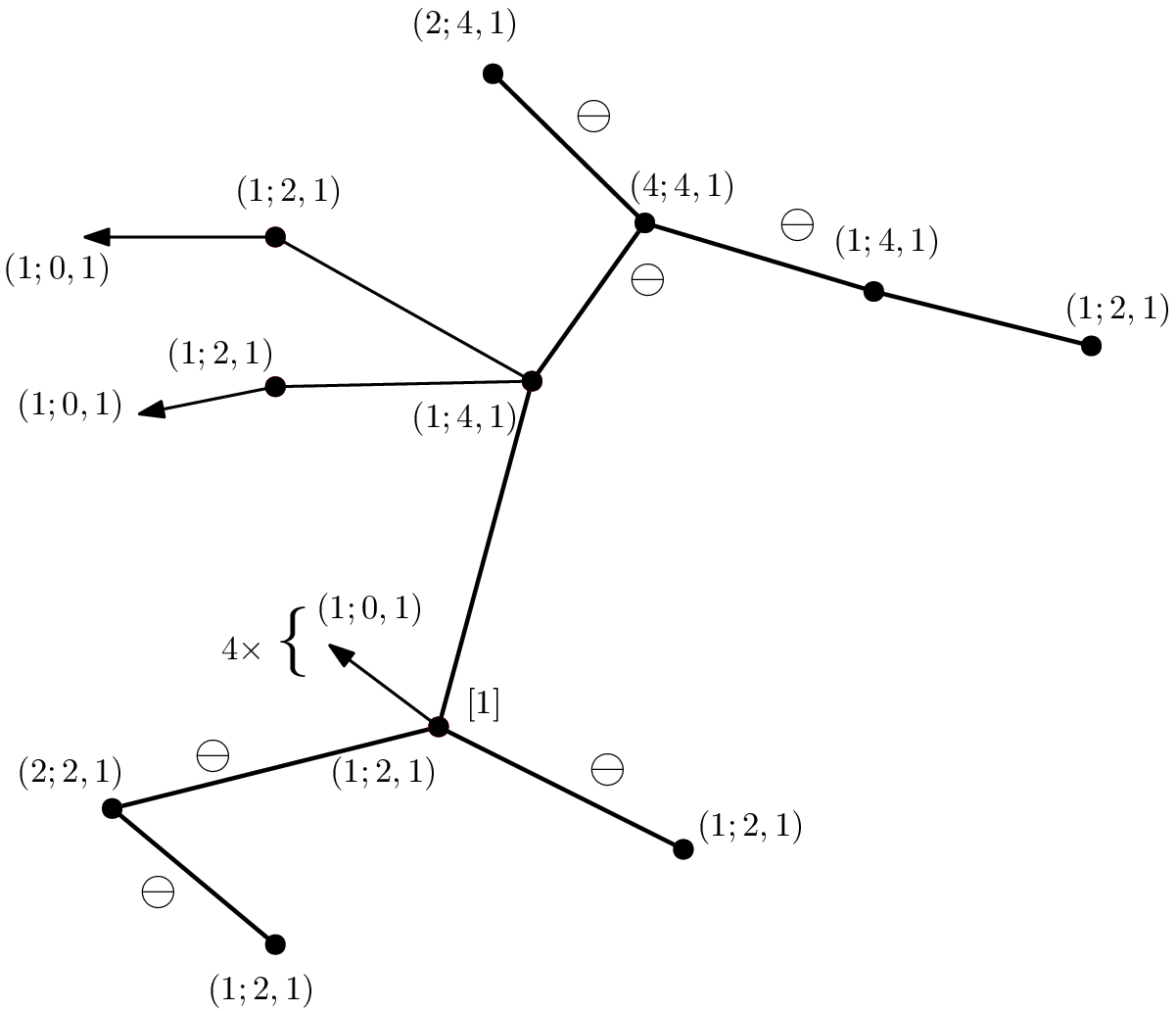}
       \caption{The graph $\Gcdt$.}
    \label{fig:gcdt}
\end{center}
\end{figure}
\end{ex}

\section{Final step of the computation and sufficiency of $\Courbest$}\label{section sufficiency}
In order to get a plumbing graph for the boundary of the Milnor fiber of $f$, one needs to modify the graph $\Gcdt$. Unlike in the general case, the necessary data for this operation is entirely encoded in $\Gcdt$. Indeed, one has the following direct consequences of our construction:
\begin{lemma}\label{lemma sufficiency}
\begin{enumerate}
\item \label{rational} Every non-rational curve of $\Courbest$ has a multiplicity decoration of the form $(1;m_2,n_2)$.
\item \label{simple cycles} There is no cycle in $\Gcdt$ made of curves of $\Dfex \cap \Do$, so each cycle in the graph $\Gcdt$ contains a vertex having a multiplicity decoration of the form $(1;m_2,n_2)$.
\end{enumerate}
\end{lemma}

\begin{proof}
\begin{enumerate}
\item The only non-rational curves of $\Gcdt$ are strict tranform curves, having therefore a first multiplicity decoration equal to $1$.
\item The curves of $\Dfex\cap \Do$ are organized in chains, one for each $1$-dimensional orbit of $X_\sigma$ contained in $\Vf$, and these chains have two-by-two empty intersection.
\end{enumerate}
\end{proof}

In the language developped in \cite{cur19.2}, point \ref{rational} implies that the data of the so-called switches is not required, and point \ref{simple cycles} implies that a covering data of $\Gcdt$ determines a unique graph, see \cite[Theorem 1.20]{Nem00}.

In other words, to obtain a plumbing graph for $\partial F$, the data contained in $\Gcdt$ are sufficient. As in \cite[10.2]{NemSzi12} (see also \cite{cur19.2} for more details), a plumbing graph for the boundary $\partial F$ of the Milnor fiber of $f$ is obtained from $\Gcdt$ by the following two steps:

\subsection{The multiplicity graph $\Gmult$}\label{subs mult graph}
Denote $\Gmult$ the graph obtained from $\Gcdt$ in the following way:

\begin{enumerate}
\item \label{cov vertex} Let $v_C$ be a vertex of $\Gcdt$ corresponding to a curve $C$ of genus $g$, such that the union of $v_C$ and the adjacent vertices is of the form

\begin{center}
\begin{tikzpicture}

\filldraw 
(-0.7,0) circle (0pt) node[above] {$(m_1;m_s,n_s)$} node[below] {$[g_s]$}
(6.7,0) circle (0pt) node[above] {$(m_{s+t};m_2,n_2)$} node[below] {$[g_{s+t}]$}
(6,0) circle (2pt) node {}
(0,0) circle (2pt) node {}
(3,0.8) circle (2pt) node {}
(3,1) circle (0pt) node[above] {$(m_1;m_2,n_2)$} 
(3,0.8) circle node[below] {$[g]$}
(0,1.6) circle (2pt) node[above] {$(m_1;m_3,n_3)$} node[below] {$[g_3]$}
(6,1.6) circle (2pt) node[above] {$(m_{s+1};m_2,n_2)$} node[below] {$[g_{s+1}]$}
(0.5,0.9) circle (0pt) node {$\vdots$}
(5.5,0.9) circle (0pt) node {$\vdots$}
;
\draw (0,0) -- node[above]{$\plus$} node[below]{$\nu_s$} (3,0.8);
\draw (0,1.6) -- node[above]{$\plus$} node [below]{$\nu_3$}(3,0.8);
\draw (6,0) -- node[above]{$\minus$} node[below]{$\nu_{s+t}$} (3,0.8);
\draw (6,1.6) -- node[above]{$\minus$} node[below]{$\nu_{s+1}$}(3,0.8);

\end{tikzpicture}

\end{center}

where $\nu_i$ indicates that an edge is repeated $\nu_i$ times. Then replace $v_C$ by the union of $n_C:=gcd(m_1,\cdots,m_{s+t})$ vertices, each one of them decorated by the same genus $g_C$ verifying 
$$ n_C\cdot (2-2\cdot g_C)=\left(2-2g-\sum\limits_{i=3}^{s+t}\nu_i\right)\cdot gcd(m_1,m_2) + \sum\limits_{i=3}^{s+t}gcd(m_1,m_2,m_i)\cdot \nu_i,$$

 and a mutliplicity decoration given by $\mu_C:=\dfrac{m_1\cdot n_2}{gcd(m_1,m_2)}$.
\item \label{cov edge} \begin{itemize}
\item An edge of type $\plus$, of the form \begin{center}
\begin{tikzpicture}

\filldraw 
(0,0.3) circle (0pt) node[above] {$(m_2;m_1,n_1)$} 
(3,0.3) circle (0pt) node[above] {$(m_2;m_3,n_3)$}
(0,0) circle (2pt) node[below] {} 
(3,0) circle (2pt) node[below] {}
(0,-0.2) circle (0pt) node[below] {$[\chi]$} 
(3,-0.2) circle (0pt) node[below] {$[\chi']$}
(-0.2,0) circle (0pt) node[left] {$v_{\Ct}$} 
(3.2,0) circle (0pt) node[right] {$v_{\Cpt}$} 
;
\draw (0,0) -- node[above]{$\plus$} (3,0);

\end{tikzpicture}
\end{center}

is replaced by $d:=gcd(m_1,m_2,m_3)$ identical strings of type $$Str^\plus\left(\frac{m_2}{d};\frac{m_1}{d},\frac{m_3}{d}|0;n_1,n_3\right).$$

\item An edge of type $\minus$, of the form 

\begin{center}
\begin{tikzpicture}

\filldraw 
(0,0.3) circle (0pt) node[above] {$(m_2;m_1,n_1)$} 
(3,0.3) circle (0pt) node[above] {$(m_3;m_1,n_1)$}
(0,0) circle (2pt) node[below] {} 
(3,0) circle (2pt) node[below] {}
(0,-0.2) circle (0pt) node[below] {$[\chi]$} 
(3,-0.2) circle (0pt) node[below] {$[\chi']$}
(-0.2,0) circle (0pt) node[left] {$v_{\Ct}$} 
(3.2,0) circle (0pt) node[right] {$v_{\Cpt}$} 
;
\draw (0,0) -- node[above]{$\minus$} (3,0);

\end{tikzpicture}
\end{center}

is replaced by $d:=gcd(m_1,m_2,m_3)$ identical strings of type $$Str^\minus\left(\frac{m_1}{d};\frac{m_2}{d},\frac{m_3}{d}|n_1;0,0\right).$$

\end{itemize}

The notations $Str$ are explained in Appendix \ref{appendix resol HJ}. In both cases, the number $d$ of new edges is a multiple of the number of new vertices, and the edges are distributed \textbf{uniformly} on the vertices.
\end{enumerate}

\begin{ex}
Figure \ref{fig:Gmult} shows the graph $\Gmult$ of our example. Here, the different bamboos are simple and do not bring any new vertices, but the upper vertex gave rise to two vertices, and the edge ending at this vertex has also produced two edges. The multiplicity decorations on $\Gmult$ correspond to the multiplicities of the pullback of the function $g$ on each irreducible surface of $(g\circ \Pi)^{-1}(0)$.

\begin{figure}
\begin{center}
\includegraphics[totalheight=7cm]{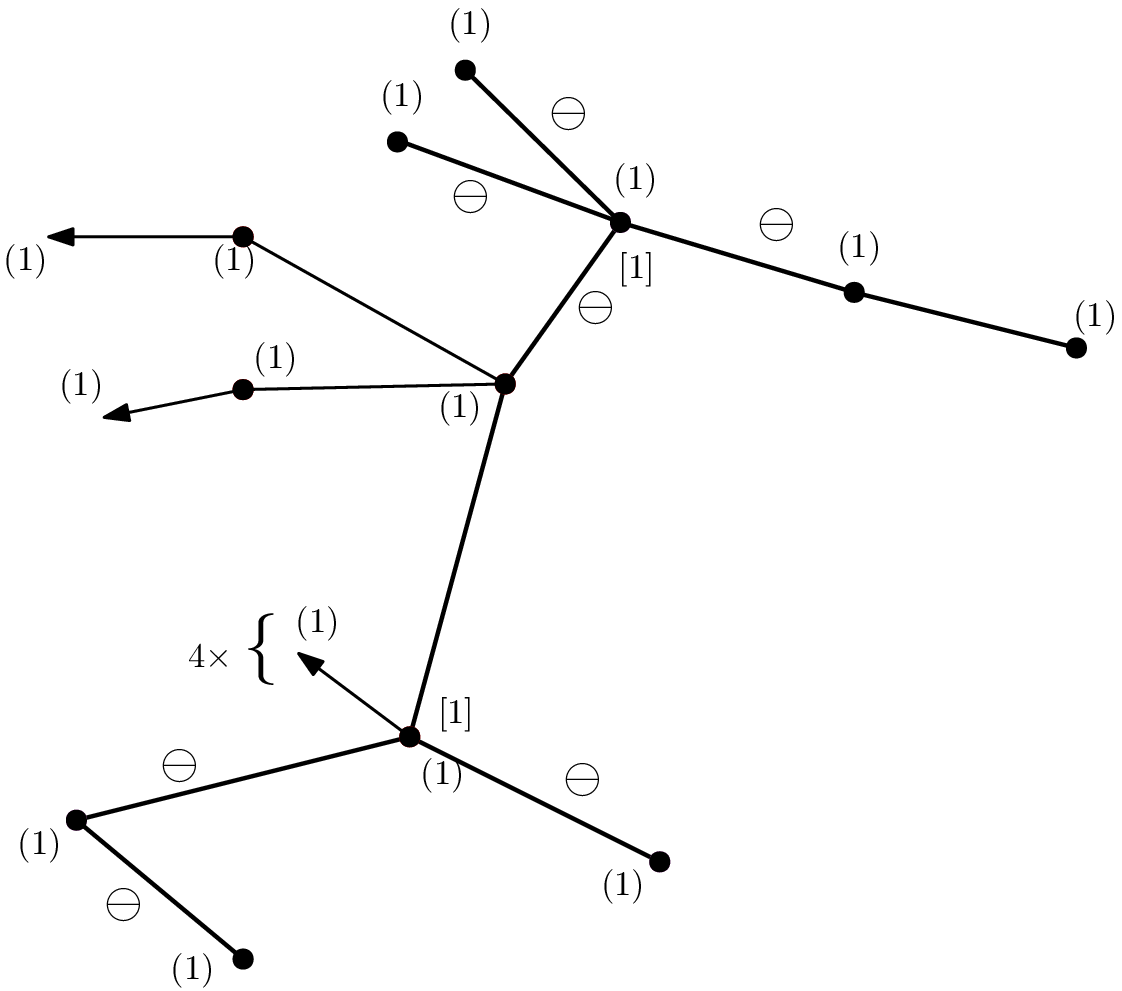}
\caption{The graph $\Gmult$.}
\label{fig:Gmult}
\end{center}
\end{figure}
\end{ex}

\subsection{The plumbing graph $\Gplomb$}\label{subs plumbing graph}

A plumbing graph $\Gamma(\partial F)$ is obtained from $\Gmult$ by replacing the multiplicity decorations $(\mu_i)$ by the self-intersections $k_i$ of the components of the preimage of the origin by $\Pi$ in $\Sb$, erasing the arrowhead vertices, and keeping the genus and edge decorations.

Let $v$ be a vertex of $\Gmult$, with multiplicity $\mu$. Let $v_1,\cdots,v_n$ be the adjacent vertices (including arrowheads), with respective multiplicities $\mu_1,\cdots,\mu_n$, and denote $\epsilon_1,\cdots,\epsilon_n$ the decorations of the corresponding edges. Then, following Lemma \ref{lemma compute self-intersections}, the self-intersection decoration $k$ of $v$ in $\Gplomb$ is obtained through the equality $$k\cdot \mu =-\sum\limits_{i=1}^{n} \epsilon_i \cdot \mu_i.$$

\begin{ex}
Figure \ref{fig:Gplomb} shows the plumbing graph for $\partial F$ obtained after applying the step described previously.

\begin{figure}
\begin{center}
\includegraphics[totalheight=7cm]{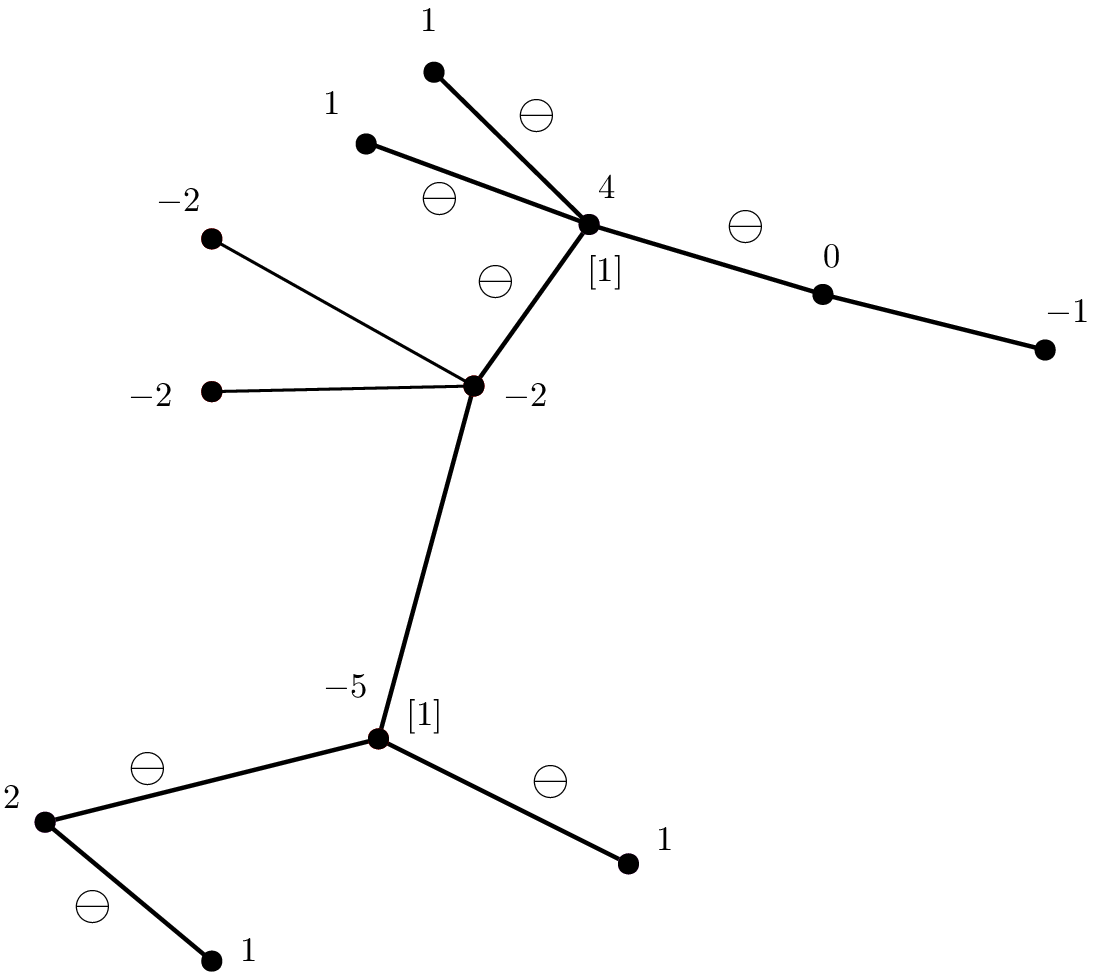}
\caption{The graph $\Gplomb$.}
\label{fig:Gplomb}
\end{center}
\end{figure}

\end{ex}

In \cite{Neu81}, the author describes the so-called \textbf{plumbing calculus}, which consists in a list of transformations one can apply on a given plumbing graph without changing the associated graph manifold. Although there are infinitely many different graphs encoding the same graph manifold, Neumann provides in \cite[Section 4]{Neu81} the definition of the \textbf{normal form} of a given plumbing graph, which is uniquely defined, as well as an algorithm to obtain it. This description is quite long, so we refer the interested reader to this work.

\begin{ex}

Figure \ref{fig:Gplombnormal} shows the normal form of this graph of Figure \ref{fig:Gplomb}. Note that, unlike in the case of isolated singularities, this graph does not have a negative definite incidence matrix.

\begin{figure}
\begin{center}
\includegraphics[totalheight=3cm]{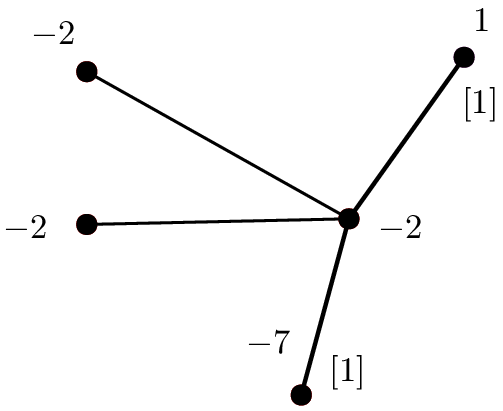}
\caption{The normal form of $\Gplomb$.}
\label{fig:Gplombnormal}
\end{center}
\end{figure}
\end{ex}

The plumbing calculus does not preserve planarity. However, the process of normalization of a given plumbing graph preserves this property. Therefore, as a consequence of our method of computation, we get:
\begin{prop}\label{prop planar}
The normal form of the plumbing graph of the Milnor fiber of a Newton non degenerate singularity of complex surface is planar.
\end{prop}

\begin{appendices}

\section{Tubular neighbourhoods and graph manifolds}
In the setting of isolated singularities of complex surfaces, graph manifolds appear naturally as boundaries of neigbourhoods of the exceptional divisors of their normal crossings resolutions, see \cite{Mum61}. In our context, they will appear as boundaries of neighbourhoods of real surfaces inside smooth real analytic $4$-manifolds, which motivates the following definition: 

\begin{defn}{\bf (Simple configuration of surfaces, and its plumbing dual graphs.)}\label{def simple conf of surfaces, plumbing graphs.}

Let $\Sb$ be a $4$-dimensional oriented real analytic manifold. A \textbf{simple configuration of compact real analytic surfaces in $\Sb$} is a subset $E\subset \Sb$ such that:
\begin{enumerate}
\item $E=\bigcup\limits_{finite}E_i$, such that each $E_i$ is an oriented closed smooth real analytic surface.
\item For all $i\neq j \neq k \neq i$, the intersection $E_i\cap E_j \cap E_k$ is empty.
\item For all $i\neq j$, the intersection $E_i \cap E_j$ is either empty or transverse. In particular, it is a finite union of points.
\end{enumerate}
In this setting, one defines a \textbf{plumbing dual graph} $\Gdp{E}{\Sb}$ of $E$ in $\Sb$ by decorating its dual graph in the following way:
\begin{enumerate}
\item Decorate each vertex $v_{E_i}$ by the self-intersection $e_i$ of $E_i$ in $\Sb$ and by the genus $[g_i]$ of the surface $E_i$. 
\item Decorate each edge of $\Gra[E]$ corresponding to the intersection point $p$ of $E_i\cap E_j$, by $\oplus$ if the orientation of $E_i$ followed by the orientation of $E_j$ is equal to the orientation of $\Sb$ at $p$, and by $\ominus$ otherwise. 
\end{enumerate}
\end{defn}

\begin{defn}\label{def rug function}(See \cite{Dur83}.) Let $E$ be a simple configuration of compact orientable real analytic surfaces in an oriented $4$-dimensional real analytic manifold $\Sb$. In this context, we call \textbf{rug function} any real analytic proper function $\rho\colon \Sb \rightarrow \R_+$ such that $\rho^{-1}(0)=E$.
\end{defn}

\begin{theorem}\label{thm corresp boundary plumbing graph}
Let $E$ be a simple configuration of compact orientable real analytic surfaces in an oriented $4$-dimensional real analytic manifold $\Sb$, that admits a rug function $\rho$ as in Definition \ref{def rug function}. Then, for $\varepsilon >0$ small enough, the boundary of the oriented $4$-manifold $\{\rho\leqslant\varepsilon\}$ is orientation-preserving homeomorphic to the graph manifold associated to the graph $\Gdp{E}{\Sb}$.
\end{theorem}

\begin{defn}
In this setting, the manifold $\{\rho\leqslant\varepsilon\}$ is called a \textbf{tubular neighbourhood} of $E$. 
\end{defn}

\begin{proof}[About the proof of Theorem \ref{thm corresp boundary plumbing graph}.]

This theorem can be seen as an extension of what is done in \cite{Mum61}, in the case of a configuration of complex analytic curves in a smooth complex surface. 

In our case, observe first that we can extend the definition of rug functions to semi-analytic functions $\rho$, and still have a unique homeomorphism type for the boundary of the neighbourhood $\{\rho\leqslant\varepsilon\}$ of $E$ for $\varepsilon>0$ small enough, following the proof of \cite[Proposition 3.5]{Dur83}. Now, one can build by hand a semi-analytic neighbourhood whose boundary is homeomorphic to the manifold $\Gdp{E}{\Sb}$. 

This is done by building a rug function for each irreducible component $E_i$ of $E$, providing a tubular neighbourhood $T_i$ of each $E_i$ whose boundary is an $\Sp^1$-bundle of Euler class $e_i$ over $E_i$. One then plumbs those bundles using appropriate normalizations of the rug functions, building a semi-analytic neighbourhood of $E$ which is homeomorphic to the desired graph manifold.
\end{proof}

\begin{rk}
Note that the decorations on the edges of the graph $\Gdp{E}{\Sb}$ depend on the orientations of the surfaces $E_i$. However, if the surfaces $E_i$ are only orientable, the different possible plumbing dual graphs still encode the same graph manifold, see move {\bf [R0]} of the plumbing calculus in \cite{Neu81}.
\end{rk}

We end this section with the tool that we use to compute the self-intersections of the irreducible components of $E$:

\begin{defn}\label{def adapted function}\index{Adapted function}
Let $E=\bigcup\limits_{finite}E_i$ be a simple configuration of compact oriented real analytic surfaces in a $4$-dimensional real analytic manifold $\Sb$. A real analytic function $g\colon \Sb \rightarrow \C$ is called \textbf{adapted} to $E$ if 
\begin{enumerate}
\item $E_{tot}:=g^{-1}(0)$ is a simple configuration of orientable real analytic surfaces, not necessarily compact, such that $E_{tot}\supset E$.
\item \begin{enumerate}
\item \label{mutiplicity gen} For any component $E_i$ of $E_{tot}$, $\forall ~ p\in E_i \setminus \bigcup\limits_{j \neq i}E_j$, there is a neighbourhood $U_p$ of $p$ in $\Sb$ and complex coordinates $(x_p,y_p)$ on $U_p$ such that $U_p\cap E_i=\{x_p=0\}$ and $n_i\in \N^*$ such that $$g=x_p^{n_i} \cdot \varphi$$ where $\varphi\colon U_p \rightarrow \C$ is a unit at $p$.

\item \label{multiplicity double} For any components $E_i$ of $E$, $E_j$ of $E_{tot}$, $\forall ~ P\in E_i\cap E_j$, there is a neighbourhood $U_p$ of $p$ in $\Sb$ and complex coordinates $(x_p,y_p)$ on $U_p$ such that $U_p\cap E_i=\{x_p=0\}$, $U_p\cap E_k=\{y_p=0\}$,  and $n_i, n_k\in \N^*$ such that $$g=x_p^{n_i}y_p^{n_j} \cdot \varphi$$ where $\varphi\colon U_p \rightarrow \C$ is a unit at $p$.
\end{enumerate}
\end{enumerate}
\end{defn}

\begin{defn}\label{def multiplicity}
In this setting, the integer $n_i$ of point \ref{mutiplicity gen} of Definition \ref{def adapted function}, independent of the point $p\in E_i \setminus \bigcup\limits_{j \neq i}E_j$, is called \textbf{the multiplicity of $g$ on $E_i$}, denoted $m_{E_i}(g)$.
\end{defn}

\begin{lemma}{\bf (Computing self-intersections.)}\label{lemma compute self-intersections}

Let $\Sb,E,g,E_{tot}$ be as in Definition \ref{def adapted function}. Let $E^{(1)}$ be an irreducible component of $E$. Then the self-intersection $e^{(1)}$ of the surface $E^{(1)}$ in $\Sb$ verifies the following condition: let $p_1,\cdots,p_n$ be the intersection points of $E^{(1)}$ with other components of $E_{tot}$, $p_j\in E_j\cap E^{(1)}$, where the same component may appear several times. Then $$n^{(1)}\cdot e^{(1)}+\sum\limits_{i=1}^{n} \epsilon_i \cdot n_i=0$$ where $\epsilon_i\in \{-1,+1\}$ refers to the sign associated to the intersection $p_i$ in the following sense: if $p\in E_i\cap E_j$, associate $+1$ to $p$ if and only if the combination of the orientations of $E_i$ and $E_j$ at $p$ provides the ambient orientation of $\Sb$.
\end{lemma}
\begin{proof}
The proof follows the standard argument in the holomorphic category. The difference of the two members of the equation is the intersection number of $E^{(1)}$ with the cycle defined by $g=0$. This cycle is homologous with that defined by a nearby level of $g$, which does not meet $E^{(1)}$ anymore. The intersection number being invariant by homology, one gets the desired result. 

In order to make this argument rigorous, one has to work in convenient tubular neighborhoods of $E$ and to look at the cycles defined by the levels of $g$ in the homology of the tube relative to the boundary.
\end{proof}

\section{Hirzebruch-Jung strings}\label{appendix resol HJ}

We introduce here the Hirzebruch-Jung strings that occur in the final step of our resolution of $(\Skt,0)$. One can consult \cite[III.5]{BPV84}, \cite[4.3]{NemSzi12} or \cite{Pop07} for more details.

For $a,b,c\in \N$, $(a,b,c)$ denotes $gcd(a,b,c)$. Suppose that $d:=(a,b,c)=1$. In these conditions, the normal surface $(V,p):=\left(\{x^a=y^bz^c\},0\right)^{norm}$ is an isolated singularity of complex surface. We describe here a graph of resolution of $(V,p)$, decorated with the self-intersections of the irreducible components of the exceptional divisor and the multiplicities of the pullback of the function $g(x,y,z)=x^{n_1}y^{n_2}z^{n_3}$.

Denote $\delta:=\dfrac{a}{(a,b)(a,c)}$, and denote by $\alpha$ the unique integer in $[0,\delta-1]$ such that $$a~|~\alpha c (a,b)+b(a,c).$$

Let $$\frac{\delta}{\alpha}=k_1-\cfrac{1}{k_2-\cfrac{1}{\cdots-\cfrac{1}{k_l}}}$$ be the negative fraction expansion of $\delta / \alpha$, $k_i\geqslant 2$, $k_i\in \N$.

Define $\mu_{l+1}:=\dfrac{b\cdot n_1+a\cdot n_2}{(a,b)},\mu_{0}:=\dfrac{c\cdot n_1+a\cdot n_3}{(a,c)},\mu_1:=\dfrac{\alpha\cdot \mu_0+\mu_{l+1}}{\delta}$ and define $\mu_2,\cdots,\mu_l$ by the relation $$\mu_{i+1}=k_i\cdot\mu_i-\mu_{i-1}$$
for any $1\leqslant i \leqslant l$. Then the graph of Figure \ref{fig:HJbambouxyz} is a graph of resolution of $(V,p)$. Left and right-hand arrows represent, respectively, the pullbacks of the $z$ and the $y$-axes. The numbers between parentheses are the multplicities of the pullback of the function $g(x,y,z)=x^{n_1}y^{n_2}z^{n_3}$ on the irreducible components of its total transform.

\begin{figure}[h]
\begin{center}

        \includegraphics[totalheight=1.4cm]{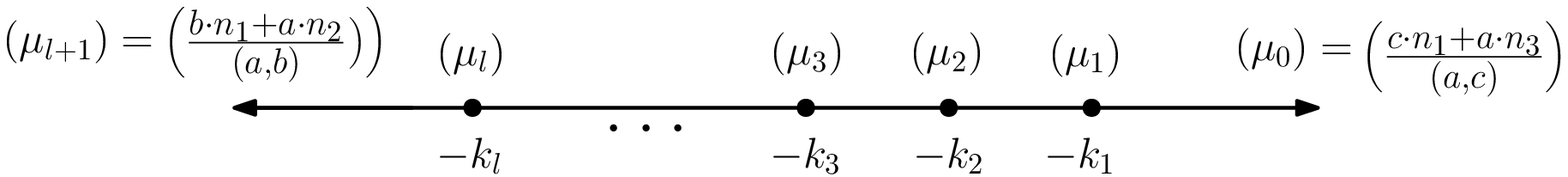}
    \caption{The string $Str(a;b,c|n_1;n_2,n_3)$.}
    \label{fig:HJbambouxyz}
\end{center}
\end{figure}

We denote by $Str^\oplus(a;b,c|n_1;n_2,n_3)$, resp. $Str^\ominus(a;b,c|n_1;n_2,n_3)$ the chain of Figure \ref{fig:HJbambouxyz} with each edge decorated by $\plus$, resp. $\minus$, and the self-intersection decorations removed.

\end{appendices}

\bibliographystyle{alpha}
\bibliography{biblio} 

\Adresses
\end{document}